\documentclass{amsart}

\usepackage{amssymb,mathrsfs,amscd,graphicx, array}
\allowdisplaybreaks
\usepackage{comment}
\usepackage[nocompress, noadjust]{cite}  

\usepackage[shortlabels]{enumitem}

\usepackage{mathtools}

\usepackage[OT2,T1]{fontenc}
\DeclareSymbolFont{cyrletters}{OT2}{wncyr}{m}{n}
\DeclareMathSymbol{\Sha}{\mathalpha}{cyrletters}{"58}

\makeatletter
\newcommand*{\Relbarfill@}{\arrowfill@\Relbar\Relbar\Relbar}
\newcommand*{\xeq}[2][]{\ext@arrow 0055\Relbarfill@{#1}{#2}}
\makeatother

\usepackage{tikz}
\usetikzlibrary{matrix,arrows,cd}

\usepackage[
  pdfencoding=auto, 
  psdextra,
  hypertexnames=false,
]{hyperref}
\usepackage{letltxmacro}
\LetLtxMacro{\oldsqrt}{\sqrt}
\renewcommand{\sqrt}[2][]{\,\oldsqrt[#1]{#2}\,}

\makeatletter
\def\@tocline#1#2#3#4#5#6#7{\relax
  \ifnum #1>\c@tocdepth 
  \else
    \par \addpenalty\@secpenalty\addvspace{#2}%
    \begingroup \hyphenpenalty\@M
    \@ifempty{#4}{%
      \@tempdima\csname r@tocindent\number#1\endcsname\relax
    }{%
      \@tempdima#4\relax
    }%
    \parindent\z@ \leftskip#3\relax \advance\leftskip\@tempdima\relax
    \rightskip\@pnumwidth plus4em \parfillskip-\@pnumwidth
    #5\leavevmode\hskip-\@tempdima
      \ifcase #1
       \or\or \hskip 1em \or \hskip 2em \else \hskip 3em \fi%
      #6\nobreak\relax
    \dotfill\hbox to\@pnumwidth{\@tocpagenum{#7}}\par
    \nobreak
    \endgroup
  \fi}
\makeatother

\usepackage{bbm}
\makeatletter

\def\greekbolds#1{%
 \@for\next:=#1\do{%
    \def\X##1;{%
     \expandafter\def\csname V##1\endcsname{\boldsymbol{\csname##1\endcsname}}
     }
   \expandafter\X\next;
  }
}

\greekbolds{alpha,beta,iota,gamma,lambda,nu,eta,Gamma,varsigma}

\def\make@bb#1{\expandafter\def
  \csname bb#1\endcsname{{\mathbb{#1}}}\ignorespaces}

\def\make@bbm#1{\expandafter\def
  \csname bb#1\endcsname{{\mathbbm{#1}}}\ignorespaces}

\def\make@bf#1{\expandafter\def\csname bf#1\endcsname{{\bf
      #1}}\ignorespaces} 

\def\make@gr#1{\expandafter\def
  \csname gr#1\endcsname{{\mathfrak{#1}}}\ignorespaces}

\def\make@scr#1{\expandafter\def
  \csname scr#1\endcsname{{\mathscr{#1}}}\ignorespaces}

\def\make@cal#1{\expandafter\def\csname cal#1\endcsname{{\mathcal
      #1}}\ignorespaces} 

\def\do@Letters#1{#1A #1B #1C #1D #1E #1F #1G #1H #1I #1J #1K #1L #1M
                 #1N #1O #1P #1Q #1R #1S #1T #1U #1V #1W #1X #1Y #1Z}
\def\do@letters#1{#1a #1b #1c #1d #1e #1f #1g #1h #1i #1j #1k #1l #1m
                 #1n #1o #1p #1q #1r #1s #1t #1u #1v #1w #1x #1y #1z}
\do@Letters\make@bb   \do@letters\make@bbm
\do@Letters\make@cal  
\do@Letters\make@scr 
\do@Letters\make@bf \do@letters\make@bf   
\do@Letters\make@gr   \do@letters\make@gr
\makeatother

\newcommand{\wt}{\widetilde}

\newcommand{\Lsymb}[2]{\genfrac{(}{)}{}{}{#1}{#2}}  

\DeclareMathSymbol{\twoheadrightarrow} {\mathrel}{AMSa}{"10}

\DeclareMathOperator{\id}{id}

\DeclareMathOperator{\pr}{pr}
\DeclareMathOperator{\rk}{rank}

\DeclareMathOperator{\diag}{diag}

\DeclareMathOperator{\Ann}{Ann}

\DeclareMathOperator{\Bil}{Bil}
\DeclareMathOperator{\Latt}{Latt}

\DeclareMathOperator{\Aut}{Aut}
\DeclareMathOperator{\End}{End}
\DeclareMathOperator{\Hom}{Hom}

\DeclareMathOperator{\Gal}{Gal}
\DeclareMathOperator{\Mat}{Mat}

\DeclareMathOperator{\Tr}{Tr}
\DeclareMathOperator{\Nm}{N}  

\DeclareMathOperator{\GL}{GL}

\DeclareMathOperator{\SO}{SO}

\newcommand{\Gin}{\Gamma_{\mathrm{in}}}
\newcommand{\Grs}{\Gamma_{\mathrm{rs}}}
\newcommand{\Pin}{\mathscr{P}_\mathrm{in}}

\def\Fpbar{\overline{\bbF}_p}
\def\Fp{{\bbF}_p}
\def\Fq{{\bbF}_q}

\newcommand{\<}{\langle}   
\renewcommand{\>}{\rangle} 

\newcommand{\isoto}{\stackrel{\sim}{\longrightarrow}}
\newcommand{\embed}{\hookrightarrow}

\newcommand{\Z}{\mathbb Z}
\newcommand{\Q}{\mathbb Q}

\newcommand{\F}{\mathbb F}

\newcounter{thmcounter} 
\numberwithin{thmcounter}{section}  
\newtheorem{thm}[thmcounter]{Theorem}
\newtheorem{lem}[thmcounter]{Lemma}
\newtheorem{cor}[thmcounter]{Corollary}
\newtheorem{prop}[thmcounter]{Proposition}

\theoremstyle{definition}
\newtheorem{defn}[thmcounter]{Definition}
\newtheorem{ex}[thmcounter]{Example}

\newtheorem{rem}[thmcounter]{Remark}

\numberwithin{equation}{section}
\numberwithin{figure}{section}
\numberwithin{table}{section}

\newtheoremstyle{notitle}  
  {}
  {}
  {\itshape}
  {}
  {}
  {\ }
  {.5em}
  {}
\theoremstyle{notitle}

\title[Superspecial abelian varieties]{Polarized superspecial abelian varieties over $\Fp$ via hermitian lattices}

 \author{Yucui Lin}

\address{(Lin) School of Mathematics and Statistics, Wuhan University, Luojiashan, 430072, Wuhan, Hubei, P.R. China}   

\email{yucui.lin@whu.edu.cn}

 \author{Jiangwei Xue}

\address{(Xue) School of
  Mathematics and Statistics, Wuhan University, Luojiashan, 430072,
  Wuhan, Hubei, P.R. China}

\email{xue\_j@whu.edu.cn}

\author{Chia-Fu Yu}

\address{(Yu) Institute of Mathematics,
  Academia Sinica and NCTS,  No. 1, Sec. 4, Roosevelt Road, Taipei 10617, TAIWAN.}
 \email{chiafu@math.sinica.edu.tw}

\begin{document}
\date{\today} 

\begin{abstract}
   We study the set of isomorphism classes of polarized superspecial abelian varieties $(A,\lambda)$ of a 
   fixed dimension over $\mathbb{F}_p$ with Frobenius endomorphism $\pi_A=\sqrt{-p}$ and $\ker \lambda =\ker \pi_A$. This set plays an important role in the geometry of the supersingular locus, and the generalizations of Deuring's $2T-H$ Theorem by Ibukiyama and Katsura. We determine  when this set is nonempty and classify its genera. Our method reduces the problems of superspecial abelian varieties to those of certain hermitian lattices by the lattice description established by Jordan et.~al and Ibukiyama--Karemaker--Yu, and we treat these problems on the lattices concerned by arithmetic methods.    
\end{abstract}

\maketitle

\section{Introduction}
Throughout this paper, we fix a prime number $p$. Let $n\ge 1$ be a positive integer and $\Lambda_n$ be the set of isomorphism classes of
$n$-dimensional principally 
polarized superspecial abelian varieties $(A,\lambda)$ over
$\Fpbar$, called the \emph{principal genus} of $n$-dimensional 
polarized superspecial abelian varieties. Recall that an abelian variety over $\Fpbar$ is said to be 
\emph{superspecial} if
it is isomorphic to a product of supersingular elliptic curves. 
The Galois group $\Gamma:=\Gal(\Fpbar/\Fp)$ acts stably on the finite set 
$\Lambda_n$ and we denote by 
\[ \Lambda_n(\Fp):=\Lambda_n^\Gamma \]
 the subset
of fixed elements. Then $\Lambda_n(\Fp)\subseteq \Lambda_n$ 
is the subset consisting of all members
$(A,\lambda)\in \Lambda_n$ that admit a model defined over
$\Fp$ (\cite{Ibukiyama-Katsura-1994}, cf.~\cite[Lemma 5.4]{Yu:Indiana-ArithSSlocus}). 

In \cite{Ibukiyama-Katsura-1994} Ibukiyama and Katsura introduce the \emph{type number}
of $\Lambda_n$, denoted $T(\Lambda_n)$, and prove the relation
\begin{equation}
  \label{eq:1.1}
  2T(\Lambda_n)-|\Lambda_n|=|\Lambda_n(\Fp)|.
\end{equation}
They further produce an explicit formula for $|\Lambda_2(\Fp)|$. 
When $n=1$, $\Lambda_1$ reduces to the set of isomorphism classes
of supersingular elliptic curves, and $T(\Lambda_1)$ recovers the type
number of the definite quaternion $\Q$-algebra $B_{p,\infty}$ ramified precisely
at $p$ and $\infty$, in which case Equation~\eqref{eq:1.1} is a classical result
of Deuring \cite{Deuring1950}.

Assume that $n$ is even, and let $\Sigma_n$ denote the set of isomorphism
classes of $n$-dimensional polarized superspecial abelian varieties $(A,\lambda)$ over
$\Fpbar$ such that $\ker \lambda=A[F]$. The set $\Sigma_n$ is  called the \emph{non-principal genus}
of $n$-dimensional  polarized superspecial abelian varieties. Here
$F:A\to A^{(p)}:=X\otimes_{\Fpbar, {\rm Frob}} \Fpbar$ is the relative
Frobenius morphism and $A[F]$ denotes the kernel of $F$. Note that the assumption $n$ being even is necessary for $\Sigma_n$ to be non-empty, as the condition $\ker \lambda=A[F]$ implies that $n=\dim A$ is even by the Riemann-Roch Theorem \cite[\S16]{mumford:av}. 
As shown by Li and Oort~\cite{li-oort}, the set $\Sigma_n$ parametrizes irreducible
components of the moduli space $\calS_n$ over $\Fpbar$ of
$n$-dimensional principally polarized supersingular abelian varieties. The set $\Sigma_n(\Fp)$ of fixed elements by $\Gamma$ then parametrizes those irreducible
components that are defined over $\Fp$. 

Ibukiyama~\cite{Ibukiyama:Osaka2018} proves the same relation
for the non-principal genus set $\Sigma_n$: 
\begin{equation}
  \label{eq:1.2}
  2T(\Sigma_n)-|\Sigma_n|=|\Sigma_n(\Fp)|.
\end{equation} 
In~\cite{Ibukiyama:tohoku2019} he also produces an explicit formula for $T(\Sigma_2)$ by Asai's class number formula for $\SO_5$~\cite{Asai}, and hence obtains an explicit formula for $|\Sigma_2(\Fp)|$ from the relation~\eqref{eq:1.2}.  

In~\cite{Yu:Indiana-ArithSSlocus} the last named author provides a geometric and adelic approach for proving the results \eqref{eq:1.1} and~\eqref{eq:1.2} 
by describing the Galois action on the sets $\Lambda_n$ and
$\Sigma_n$, which can be regarded as an explicit reciprocity law. This approach
requires the existence of a base point $(A_0,\lambda_0)$ defined over
$\Fp$, which is obvious for the principal genus, but not so clear for the
non-principal genus. As a result, he establishes the relation \eqref{eq:1.1} unconditionally, and also for \eqref{eq:1.2} under the existence assumption of a non-principal superspecial point $(A_0,\lambda_0)$ over
$\Fp$ whose Frobenius endomorphism $\pi_{A_0}$ satisfies $\pi_{A_0}^2=-p$. Later he~\cite{yu:field-of-def-2017} shows the existence of a non-principal superspecial point $(A'_0,\lambda'_0)$ over $\Fp$ whose Frobenius endomorphism $\pi_{A_0'}$ satisfies $(\pi_{A_0'})^2=p$ and hence establishes the relation~\eqref{eq:1.2} unconditionally as well.  
However, the problem for the existence of such a non-principal superspecial abelian variety $(A_0,\lambda_0)$ over $\Fp$ with $\pi_{A_0}^2=-p$ remains open up to now. \\

The first main result of this paper is to determine when there is an object $(A,\lambda)$ in $\Sigma_n(\Fp)$ that admits a model over $\Fp$ with Frobenius endomorphism $\sqrt{-p}$. 
Let us denote the set of such members $(A,\lambda)\in \Sigma_n(\Fp)$ by $\Sigma_n(\sqrt{-p})$. 
To study the above problem, we consider the set 
$\wt\Sigma_n(\sqrt{-p})$ of $\Fp$-isomorphism classes of polarized superspecial abelian varieties $(A,\lambda)$ of dimension $n$ over $\Fp$ such that $\pi_{A}^2=-p$ and $\ker \lambda=A[\pi_A]$.
Note that $\Sigma_n(\sqrt{-p})$ is non-empty if and only if $\wt\Sigma_n(\sqrt{-p})$ is non-empty, as the natural map $\wt\Sigma_n(\sqrt{-p})\to \Sigma_n(\sqrt{-p})$ is surjective; see~\cite[Lemma~2.1]{yu:field-of-def-2017}. Thus, the problem is reduced to classifying the members in $\wt\Sigma_n(\sqrt{-p})$.

We say that two polarized abelian varieties $(A_1,\lambda_1)$ and $(A_2,\lambda_2)$ over a field $k$ are \emph{isogenous}, denoted $(A_1,\lambda_1)\sim (A_2, \lambda_2)$, if there is a quasi-isogeny $\alpha:A_1 \to A_2$ over $k$ such that $\alpha^*\lambda_2=\lambda_1$.

\begin{thm}\label{thm:main.1} For any positive even integer $n$, we have $\wt\Sigma_n(\sqrt{-p})=\emptyset$ if and only if $p\equiv 7\pmod{8}$ and $n\equiv 2\pmod{4}$. In other words,
the set $\wt\Sigma_n(\sqrt{-p})$ is non-empty if and only if one of the following mutually exclusive conditions holds:
\begin{itemize}
    \item [{\rm (a)}] $p\not \equiv 3 \pmod 4$;
    \item [{\rm (b)}] $p\equiv 3 \pmod 4$ and $4|n$;
    \item [{\rm (c)}] $p\equiv 3 \pmod 8$ and $n\equiv 2 \pmod 4$.
\end{itemize}   
Moreover, in each case, any two members $(A_1,\lambda_1)$ and $(A_2, \lambda_2)$ in $\wt\Sigma_n(\sqrt{-p})$ are isogenous.
\end{thm}

A basic step of the proof is to observe that every member of $\wt\Sigma_n(\sqrt{-p})$ is a polarized abelian variety over $\Fp$  isogenous to a power of an elliptic curve $E$ over $\Fp$ with $\pi_E^2=-p$. Thus, we can use  the very recent result of lattice descriptions of such geometric objects due to Ibukiyama-Karemaker-Yu~\cite[Corollary~4.6]{Ibukiyama-Karemaker-Yu-2025}, which is in turn built upon the work of Jordan et al.~\cite{JKPRST}. This reduces problems on $\wt\Sigma_n(\sqrt{-p})$ to those on a class of positive-definite hermitian $\Z[\sqrt{-p}]$-lattices, and we treat such problems using arithmetic methods. This lattice description will be recalled in Section~\ref{sec:lattice-descr}.
We refer to \cite{BKM, karemaker-pries, centeleghe-stix, centeleghe-Stix:2} for more studies on (polarized or unpolarized) abelian varieties over finite fields. 

Let $K=\Q(\sqrt{-p})$ and $O_K$ be the ring of integers of $K$. Let $\wt\Sigma_{n,1}(\sqrt{-p})\subseteq \wt\Sigma_n(\sqrt{-p})$ be the subset consisting of all members $(A,\lambda)$ of $\wt\Sigma_n(\sqrt{-p})$ such that $O_K\subseteq \End(A)$. If $\Z[\sqrt{-p}]\subsetneq O_K$, that is, $p\equiv 3 \pmod 4$, let $\wt\Sigma_{n,2}(\sqrt{-p})$ be the complement of $\wt\Sigma_{n,1}(\sqrt{-p})$ in $\wt\Sigma_n(\sqrt{-p})$. Namely, we have 
\[ \wt\Sigma_n(\sqrt{-p})=\wt\Sigma_{n,1}(\sqrt{-p})\,\bigsqcup\, \wt\Sigma_{n,2}(\sqrt{-p}).\]

Theorem~\ref{thm:main.1} follows immediately from the following refinement.

\begin{thm}
    \label{thm:main.2}
    Let the notation be as above.
\begin{enumerate}
    \item [{\rm (1)}] The set $\wt\Sigma_{n,1}(\sqrt{-p})$ is non-empty if and only if one of conditions {\rm (a)} and {\rm (b)} in Theorem~\ref{thm:main.1} holds.
    \item [{\rm (2)}] When $p\equiv 3\pmod{4}$,  the set $\wt\Sigma_{n,2}(\sqrt{-p})$is non-empty if and only if one of conditions {\rm (b)} and {\rm (c)} in Theorem~\ref{thm:main.1} holds. 
\end{enumerate}    
\end{thm}

The second part of this paper focuses on the classification of polarized abelian varieties in $\wt\Sigma_n(\sqrt{-p})$. This can be reduced to classify smaller subsets called genera. Two polarized abelian varieties $(A_1,\lambda_1)$ and $(A_2,\lambda_2)$ over a finite field $\Fq$ are said to \emph{belong to the same genus} if they are isogenous and for every prime $\ell$ including $p$, their associated polarized $\ell$-divisible groups $(A_1,\lambda_1)[\ell^\infty]$ and $(A_2,\lambda_2)[\ell^\infty]$ are isomorphic over $\Fq$. 
Our second main result classifies the genera in $\wt\Sigma_n(\sqrt{-p})$, which paves the way for  more precise size estimates of $\wt\Sigma_n(\sqrt{-p})$.  Using the result \cite[Theorem 2.2]{yu:smf} (cf.~\cite[Theorem 5.10]{xue-yu:counting-av}), we can express each genus $\calG$ as a double coset space attached to the automorphism group scheme $G_x$ over $\Z$ of any base point $x=(A,\lambda)\in \calG$. Then we can compute the associated mass of each genus $\calG$, and estimate the size of $\calG$. See \cite{Hashimoto-Ibukiyama-1,Hashimoto-Koseki, karemaker2020mass, xue-yu:ppas, terakado-yu:basic-locus2022, terakado-xue-yu:2023} for some related work in this direction. 

\begin{thm}
    \label{thm:main.3}
   \begin{enumerate}
       \item[{\rm (1)}] If $p\not \equiv 3 \pmod 4$, then $\wt\Sigma_{n,1}(\sqrt{-p})$ consists of either two genera or one genus according to $4\mid n$ or not.
       \item[{\rm (2)}] Assume $p\equiv 3 \pmod 4$ and $4\mid n$.
         \begin{enumerate}
           \item[{\rm(a)}] $\wt\Sigma_{n,1}(\sqrt{-p})$ consists of one genus.
           \item[{\rm(b)}] $\wt\Sigma_{n,2}(\sqrt{-p})$ consists of $n$ genera if $p\equiv 3\pmod{8}$, and $3n/2$ genera if $p\equiv 7\pmod{8}$. 
         \end{enumerate}
       \item[{\rm (3)}] If $p\equiv 3 \pmod 8$ and $n\equiv 2 \pmod 4$, then $\wt\Sigma_{n,1}(\sqrt{-p})=\emptyset$ and $\wt\Sigma_{n,2}(\sqrt{-p})$ consists of $n/2$ genera.
   \end{enumerate} 
\end{thm}

Theorems~\ref{thm:main.2} and~\ref{thm:main.3} are derived directly from the lattice description of $\wt\Sigma_n(\sqrt{-p})$ in Proposition~\ref{prop:wtSigma} and purely arithmetic results on $\Z[\sqrt{-p}]$-hermitian lattices to be stated in Theorem~\ref{thm:main.4}. A $\Z[\sqrt{-p}]$-lattice $L$ in a nondegenerate hermitian $K$-space $(V,\varphi)$ is said to be {\it $\sqrt{-p}$-modular} if $\sqrt{-p} L^\vee=L$, where $L^\vee$ is the $\Z[\sqrt{-p}]$-valued dual lattice:
\begin{equation}\label{eqn:1.3}
   L^\vee:=\{\, x\in V\, | \, \varphi(x,L)\subseteq \Z[\sqrt{-p}]\, \}. 
\end{equation}
For each prime $\ell$, let $L_\ell\coloneqq L\otimes_\Z\Z_\ell$ be the $\ell$-adic completion of $L$. Two $\Z[\sqrt{-p}]$-lattices $L, L'$ in $(V, \varphi)$ are said to \emph{belong to the same genus} if $L_\ell$ is isometric to $L'_\ell$ for every prime $\ell$. 
   


\begin{thm}[Theorems~\ref{thm:main1} and \ref{thm:4.12}]\label{thm:main.4} Let $(V,\varphi)$ be a positive-definite hermitian $K$-space of even dimension $n$ and determinant $d\varphi\in \Q^\times_{+}/{\rm N}(K^\times)$, where ${\rm N}: K^\times \to \Q^\times$ is the norm map.

\begin{enumerate}
\item[{\rm (1)}] If there  exists a $\sqrt{-p}$-modular
   $O_K$-lattice $L$ in $(V, \varphi)$, then $d\varphi=[1]$.
\item[{\rm (2)}] Suppose $d\varphi=[1]$. Then there exists a $\sqrt{-p}$-modular
   $O_K$-lattice $L$ in $(V, \varphi)$ if and only if
   \begin{itemize}
       \item[{\rm (a)}]  $p\not\equiv 3\pmod{4}$, or
       \item[{\rm (b)}] $p\equiv 3\pmod{4}$ and $4\mid n$. 
   \end{itemize}
   \item[{\rm (3)}] Suppose that one of the conditions in \emph{(2)} holds so that there exists a $\sqrt{-p}$-modular $O_K$-lattice in $(V, \varphi)$. 
   \begin{itemize}
       \item[{\rm (a)}] If $p\not\equiv 3\pmod{4}$ and $4\mid n$, then there are two genera of $\sqrt{-p}$-modular $O_K$-lattices in $(V, \varphi)$. 
       \item [{\rm (b)}] For the remaining cases (that is, either $p\not\equiv 3\pmod{4}$ and $4\nmid n$ or condition (b) in (2) holds),  there is a unique genus of $\sqrt{-p}$-modular $O_K$-lattices in $(V, \varphi)$.
   \end{itemize}
   
  \item[{\rm (4)}] Suppose that $p\equiv 3\pmod{4}$. There exists a $\sqrt{-p}$-modular $\Z[\sqrt{-p}]$-lattice $L$ in $(V, \varphi)$ if and only if 
    \begin{enumerate}
        \item[{\rm (a)}] $p\equiv 7\pmod{8}$, $4\mid n$  and $d\varphi=[1]$; or  
        \item[{\rm (b)}] $p\equiv 3\pmod{8}$, $4\mid n$ and $d\varphi=[1]$; or  
\item[{\rm (c)}]   $p  \equiv 3\pmod{8}$, $ n\equiv 2\pmod{4}$  and $d\varphi=[2]$.  
    \end{enumerate}
    \item[{\rm (5)}] Keep the assumption in {\rm (4)} and suppose that one of the conditions in {\rm (4)} holds so that there exists a $\sqrt{-p}$-modular $\Z[\sqrt{-p}]$-lattice in $(V, \varphi)$. Then each genus of such lattices is uniquely determined by the corresponding isometric class of unimodular $\Z_2[\sqrt{-p}]$-lattices in $(V_2, \varphi_2)$ at the prime $2$.  Moreover, there are $3n/2+1$ genera in case~{\rm (a)}, $n+1$ genera in case~{\rm {(b)}}, and $n/2$ genera in case~{\rm (c)}.  
    \end{enumerate}
\end{thm}

The remaining part (Sections 3--5) of this paper
treats the hermitian $\Z[\sqrt{-p}]$-lattices concerned, which is of interest in its own right. Even though our motivation arises from the geometry, these arithmetic results can be read on their own. The method for detecting the global constraint of the lattices here can be applied to more general hermitian lattices. 

We refer to Section~\ref{sec:4} for the classification of the genera of  $\Z[\sqrt{-p}]$-lattices when $p\equiv 3\pmod 4$, which is made more complicated because of the non-maximality of $\Z[\sqrt{-p}]\subsetneq O_K$.
The proof in this case requires the classification of self-dual hermitian $\Z_2[\sqrt{-p}]$-lattices. This classification is originally given in Yu~\cite[Theorems 1.1 and 1.2]{Yu-CF-PAMQ-2012}, where his proofs in cases $p\equiv 3\pmod 8$ and $p\equiv 7\pmod 8$ are treated separately. We give a uniform proof of the required classification, which results from  an orthogonal decomposition theorem in Section~\ref{sec:5} for any self-dual hermitian lattice over an arbitrary  commutative local Bass order.

Let $Z$ be a complete discrete valuation ring with fraction field $Q$, $E$ a commutative semi-simple $Q$-algebra and $R$ be a local Bass $Z$-order of $E$. See an exposition for the definition of Bass orders and classification of lattices over a Bass order in Section~\ref{sec:5}. Suppose $E$ admits a (possibly trivial) involution $\sigma$ that leaves $R$ stable and let $V$ be a free $E$-module together with a non-degenerate sesquilinear pairing 
\[ \<~,~\>: V\times V\to E. \]
If $M$ is an $R$-lattice, then there is an ascending  chain of $Z$-orders
\[   R\subseteq R_1\subsetneq R_2 \subsetneq \cdots \subsetneq R_m\]
and an $m$-tuple $(r_1, \cdots, r_m)$ of positive integers so that there is an isomorphism of $R$-lattices 
\[ M\simeq R_1^{r_1}\oplus \dots \oplus R_m^{r_m};\]
moreover, the orders $R_i$, $1\le i\le m$, and $(r_1, \cdots, r_m)$ are uniquely determined by $M$; see \cite[Section 37,
p.~789]{curtis-reiner:1}. These orders $R_i$, for $1\le i\le m$, together with the $m$-tuple $(r_1, \cdots, r_m)$ are called the \emph{structural invariants} of $M$. 
For each $i$, the $R_i$-ideal $(R:R_i):=\{x\in E\, |\, aR_i\subseteq R\}$ is principal (Lemma~\ref{lem:rel_cond}), and we fix a generator $\alpha_i$ of $(R:R_i)$. 

\begin{thm}[Theorem~\ref{thm:orthogonal-decomp}]\label{thm:main.5}
Let $M$ be a self-dual $R$-lattice in $(V,\<~,~\>)$ with structural invariants
$R_1\subsetneq R_2\subsetneq \cdots \subsetneq
R_m$ and $(r_1, \cdots, r_m)$.  Then there exists an orthogonal decomposition of
  $M$ into $R$-sublattices
  \begin{equation}
    \label{eq:1.3}
  M=M_1\oplus \cdots\oplus M_m,\qquad \text{with}\quad \langle M_i, M_j\rangle=0, \quad \forall 1\leq i<j\leq m, 
  \end{equation}
  where each $M_i$ is a free $R_i$-lattice of rank $r_i$, and the restriction $\langle~,~\rangle_{M_i}: M_i\times
M_i\to R$ of $\langle~,~\rangle$ to $M_i$ is of the form
$\alpha_i\langle~,~\rangle_i$ for a perfect sesquilinear pairing $\langle~,~\rangle_i: M_i\times
M_i\to R_i$. 
\end{thm}

If $\langle~,~\rangle$ is $\varepsilon$-hermitian for some $\varepsilon\in O_E^\times$ with $\varepsilon\varepsilon^\sigma=1$, then Theorem~\ref{thm:main.5} produces an orthogonal decomposition $M=M_1\perp\cdots \perp M_m$ in the classical sense, that is, \eqref{eq:1.3} holds with $\langle M_i, M_j\rangle=0$ for all $i\neq j$.
This theorem reduces the classification of self-dual hermitian, symmetric or symplectic $R$-lattices $M$ to the subclass of free lattices over suitable suporders $R_i\supseteq R$, provided that the generators $\alpha_i$ of $(R:R_i)$ can be chosen $\sigma$-invariant. Thus, this result can be used to extend 
some previous results in classical literature such as Baeza \cite{Baeza-QF-semilocal}, Knebusch \cite {Knebusch-1977} or
Knus \cite{Knus-1991}, where the lattices $M$ considered are projective, to general lattices over a local Bass order as above.  
We remark that the self-dual $R_i$-lattice $M_i$ in the orthogonal decomposition \eqref{eq:1.3} does not need to be uniquely determined up to isomorphism by $M$; see Example~\ref{ex:no-witt}. \\

This paper is organized as follows. Section~2 describes the set $\wt \Sigma_n(\sqrt{-p})$ in terms of hermitian lattices based on the works \cite{JKPRST} and \cite{Ibukiyama-Karemaker-Yu-2025}. Section~3 classifies the positive-definite hermitian $O_K$-lattices whose $O_K$-valued dual lattice $L^\star$ is equal to $\sqrt{-p}^{-1} L$. Extending results of Section~3, in Section~4 we classify the positive-definite $\sqrt{-p}$-modular hermitian $\Z[\sqrt{-p}]$-lattices. Section~5 treats the proof of Theorem~\ref{thm:main.5}.

\section{Polarized abelian varieties and lattice description}
\label{sec:lattice-descr}


In this section,  we review the lattice description of polarized abelian varieties over $\Fq$ that are isogenous to a power of an elliptic curve, where $\Fq$ denotes the finite field of $q$ elements for a power $q$ of $p$. 
This is the main tool to translate problems for $\wt \Sigma_n(\sqrt{-p})$ to those of suitable hermitian lattices over an imaginary quadratic order. This method also works for  more general classes of polarized abelian varieties; see Theorem~\cite[Theorem C]{Ibukiyama-Karemaker-Yu-2025}. Our exposition is based on Ibukiyama-Karemaker-Yu~\cite[Section 4]{Ibukiyama-Karemaker-Yu-2025}.  

Let $E$ be a fixed elliptic curve over $\Fq$ with Frobenius endomorphism $\pi$ and canonical principal polarization $\lambda_E$. 
Let $R:=\End(E)$ and $B:=\End^0(E)$ denote the respective endomorphism ring and endomorphism algebra of $E$ over $\Fq$, which are equipped with the involution $a \mapsto a^\sigma$ induced by $\lambda_E$. Then $B$ is either an imaginary quadratic field or the definite quaternion $\Q$-algebra $B_{p,\infty}$, and $a\mapsto a^\sigma$ is the canonical involution. Let $\prescript{}{R}{\mathrm{Lat}}$ (resp.~$\mathrm{Lat}_R$) denote the category of left (resp.~right) $R$-lattices. To each right $R$-lattice $M$, one attaches a left $R$-lattice $\prescript{}{R}{M}:=M$ with left action given by $a \cdot m:= m a^\sigma $ for $m\in M$ and $a\in R$, and vice versa. Then the functor $M\mapsto \prescript{}{R}{M}$ induces an equivalence of categories from $\mathrm{Lat}_R$ to $\prescript{}{R}{\mathrm{Lat}}$.

Following~\cite{Ibukiyama-Karemaker-Yu-2025}, by a left (resp.~right) hermitian $R$-lattice $M$ we mean an $R$-lattice in a finitely generated left (resp.~right) hermitian $B$-space $(V,h)$. That is, $h:V\times V\to B$ is a $\Q$-bilinear pairing satisfying $h(y,x)=h(x,y)^\sigma$ and
\[ h(ax,by)=ah(x,y)b^\sigma\quad  (\text{resp.}~h(xa,yb)=a^\sigma h(x,y)b ), \quad \forall a, b \in B, \, x,y\in V.\]
If $h(M,M)\subseteq R$, we say that $M$ is integral or $R$-valued, and $h$ is integral on $M$. 
Let $\prescript{}{R}{\mathrm{Lat}}^H$ (resp.~$\mathrm{Lat}_R^H$) denote the category of positive-definite left (resp.~right) hermitian $R$-lattices.  
Similarly as before, we have an equivalence of categories from $\mathrm{Lat}_R^H$ to $\prescript{}{R}{\mathrm{Lat}}^H$.

Given an abelian variety $A$, we call an element $\lambda\in \Hom(A,A^t)\otimes \Q$ a \emph{$\Q$-polarization} on  $A$ if $N\lambda$ is a polarization for some positive integer $N$; $\lambda$ is called \emph{integral} if $N$ can be taken to $1$, that is, $\lambda$ is a polarization. A $\Q$-polarized abelian variety is an abelian variety $A$ together with a $\Q$-polarization. Let $\calA_E^{\rm pol}$ denote the category of $\Q$-polarized abelian varieties over $\Fq$ such that $A$ is isogenous to a power of $E$. Morphisms in $\calA_E^{\rm pol}$ are morphisms of abelian varieties that preserve the $\Q$-polarizations (not up to a positive scarlar of $\Q$). 

To each $\Q$-polarized abelian variety $(A,\lambda)$ in $\calA^{\rm pol}_E$, we attach a pair $(M,h)$, where 
\begin{equation}\label{eq:M}
   M:=\Hom(E,A) 
\end{equation}
is a right $R$-lattice, and 
\begin{equation}\label{eq:h}
    h=h_\lambda: M_\Q \otimes M_\Q \to B, \quad h_\lambda(m_1,m_2):=\lambda_E^{-1} m_1^t \lambda m_2
\end{equation}
is a pairing on $M_\Q:=M\otimes \Q$. Then $(M,h)$ is a positive-definite hermitian $R$-lattice. The hermitian form $h$ is integral on $M$ if and only if $\lambda$ is a polarization, and it is perfect if and only if $\lambda$ is a principal polarization; see \cite[Lemma 4.4]{Ibukiyama-Karemaker-Yu-2025}. 

\begin{thm}[{\cite[Corollary 4.6]{Ibukiyama-Karemaker-Yu-2025}}]\label{thm:IKY-JKPRST}  
Let $E$ be an elliptic curve over $\Fq$ with Frobenius endomorphism $\pi$ and endomorphism ring $R = \mathrm{End}(E)$. The functor $\Hom(E,-): \calA_E^{\rm pol} \to {\rm Lat}_R^H$ induces an equivalence of categories if and only if one of the following holds:
\begin{itemize}
    \item $E$ is ordinary and $\Z[\pi]=R$;
    \item $E$ is supersingular, $\Fq=\Fp$ and $\Z[\pi]=R$; or
    \item $E$ is supersingular, $\Fq=\F_{p^2}$ and $R$ has rank $4$ over $\mathbb{Z}$.
\end{itemize}  
\end{thm}

Recall that two $\Q$-polarized abelian varieties $(A_1,\lambda_1)$ and $(A_2,\lambda_2)$ over $\Fq$ are said to  \emph{belong to the same genus} if they are isogenous and for every prime $\ell$ including $p$, their associated $\Q_\ell$-polarized $\ell$-divisible groups $(A_1,\lambda_1)[\ell^\infty]$ and $(A_2,\lambda_2)[\ell^\infty]$ are isomorphic over $\Fq$.

Let ${\mathcal Hom}_R(-,E): \prescript{}{R}{\rm Lat}^{\rm opp} \to \calA_E$ denote the sheaf Hom functor constructed in~\cite{JKPRST}, where $\calA_E$ denotes the category of abelian varieties over $\Fq$ that are isogenous to a power of $E$.  
For each left $R$-lattice $M$, the functor ${\mathcal Hom}_R(M,E)$ attaches to each $\Fq$-scheme $S$ the abelian group $\Hom_R(M,E(S))$, and it is representable by an abelian variety $A_M$ isogenous to a power to $E$. The functor ${\mathcal Hom}_R(-,E)$ is fully faithful with inverse functor $\Hom(-,E)$. 
Let $\calA_{E,{\rm ess}}^{\rm pol}$ be the full subcategory of $\calA_E^{\rm pol}$ consisting of objects $(A,\lambda)$ whose underlying abelian variety $A$ lies in the essential image of ${\mathcal Hom}_R(-,E)$. By \cite[Proposition~4.5]{Ibukiyama-Karemaker-Yu-2025}, the functor 
\begin{equation}\label{eq:Hom}
   \Hom(E,-): \calA_{E,{\rm ess}}^{\rm pol}\to {\rm Lat}_R^H 
\end{equation}
gives an equivalence of categories.

\begin{prop}\label{prop:genera}
    Any two objects $(A_1,\lambda_1)$ and $(A_2,\lambda_2)$ in $\calA_{E,{\rm ess}}^{\rm pol}$ belong to the same genus if and only if their attached hermitian $R$-lattices $(M_1,h_1)$ and $(M_2,h_2)$ belong to the same genus. 
\end{prop}
\begin{proof}
    The equivalence of categories~\eqref{eq:Hom} gives rise to a natural isomorphism
    \begin{equation}\label{eq:ff}
        \Hom(A_1,A_2)\xrightarrow[\alpha\mapsto \alpha_*]{\sim} \Hom_R(M_1,M_2),
    \end{equation}
    which in turn induces a natural isomorphism $\Hom(A_1,A_2)\otimes \Q \isoto \Hom_B(M_{1,\Q},M_{2,\Q})$. Let $\alpha\in \Hom(A_1,A_2)\otimes \Q$ be an isomorphism, and let $\beta=\alpha_*$ be the isomorphism corresponding to $\alpha$; one has $\beta(m_1)=\alpha m_1$ for $m_1\in M_{1,\Q}$. One computes that
    \[ h_2(\beta m_1, \beta m_1')=\lambda_E^{-1} m_1^t \alpha^t \lambda_2 \alpha m_1', \quad h_1(m_1,m_1')=\lambda_E^{-1} m_1^t  \lambda_1  m_1' \]
    for $m_1,m_1'\in M_{1,\Q}$. It follows that $\alpha^t \lambda_2 \alpha=\lambda_1$ if and only if $h_2(\beta m_1,\beta m_1')=h_1(m_1, m_1')$ for all $m_1, m_1'\in M_{1,\Q}$, so as $(A_1,\lambda_1)$ and
    $(A_2, \lambda_2)$ are isogenous if and only if $(M_{1,\Q},h_1)$ and $(M_{2,\Q},h_2)$ are isomorphic.

    We denote by $\lambda_{i,\ell}: A_i[\ell^\infty]\to A_i^t[\ell^\infty]$ the isogeny induced by $\lambda_i$ for every prime $\ell$. We have natural isomorphisms 
    \[ \begin{split}
        \Hom(A_1[\ell^\infty],A_2[\ell^\infty])&\simeq \Hom(A_1,A_2)\otimes \Z_\ell \\
        &\simeq \Hom_R(M_1,M_2)\otimes \Z_\ell \\
        &=\Hom_{R_\ell}(M_{1,\ell},M_{2,\ell}), \quad M_{i,\ell}:=M_i\otimes \Z_\ell.
    \end{split}\]
    The first isomorphism follows from Tate's theorem on homomorphisms of abelian varieties over finite fields (see \cite{tate:eav} and \cite[Theorem, p.~525]{waterhouse:thesis}), and the second isomorphism is given by the isomorphism \eqref{eq:ff}. Let $\alpha\in  \Hom(A_1[\ell^\infty],A_2[\ell^\infty])$ be an isomorphism, and let $\beta\in \Hom_{R_\ell}(M_{1,\ell},M_{2,\ell})$ be the isomorphism corresponding to $\alpha$. The same computation shows $\alpha^t \lambda_{2,\ell} \alpha=\lambda_{1,\ell}$ if and only if $h_2(\beta m_1,\beta m_1')=h_1(m_1, m_1')$ for all $m_1, m_1'\in M_{1,\ell}$. Thus, $(A_1,\lambda_1)[\ell^\infty]\simeq (A_2,\lambda_2)[\ell^\infty]$ if and only if $(M_{1,\ell},h_1)\simeq (M_{2,\ell},h_2)$. This proves the proposition.  
\end{proof}


Now take $\pi=\sqrt{-p}$, $\Fq=\Fp$, $R=\Z[\sqrt{-p}]$ and $B=\Q(\sqrt{-p})$.  Waterhouse shows in his thesis \cite[Theorem~4.2(3)]{waterhouse:thesis} that there exists an elliptic curve $E$ over $\Fp$ with $\End(E)=R$. We denote the canonical polarization of $E$ by $\lambda_E$ as usual.

\begin{lem}
Let $(A,\lambda)$ be a polarized abelian variety in $\calA_E^{\rm pol}$ and $(M,h)$ the attached positive-definite hermitian $R$-lattice. 
We have $\ker \lambda =A[\pi]$ if and only if $\pi M^\vee=M$, where $M^\vee\coloneqq \{x\in M\otimes \Q\mid h(x, M)\subseteq R\}$ is the dual lattice of $M$ with respect to the hermitian form $h$.
\end{lem}
\begin{proof}
Although $R$ is commutative, we still distinguish left and right $R$-modules to keep track of the functoriality properties. 
Since $M$ is a right $R$-module, its dual module $M^t\coloneqq \Hom_R(M, R)$ is naturally a left $R$-module. As discussed before, we equip it with a right $R$-module structure as follows:
\[ (f\cdot a)(m)\coloneqq a^\sigma f(m), \qquad \forall f\in M^t, a\in R,  m\in M.\]
 From the original unpolarized version of Theorem~\ref{thm:IKY-JKPRST} (namely, \cite[Theorem~1.1]{JKPRST}), the right $R$-lattice $\Hom(E, A^t)$ attached to the dual abelian variety $A^t$ is canonically isomorphic to $M^t$, where each $\alpha\in \Hom(E, A^t)$ corresponds to 
$(\lambda_E^{-1}\alpha^t)_*\in M^t$.
The polarization $\lambda:A\to A^t$ gives rise to an $R$-linear inclusion  
$\lambda_*: M\embed M^t$, which  coincides with the map induced by the hermitian form $h$ sending each $m_1\in M$ to $h(m_1, -)\in M^t$. 
Observe that $\ker \lambda =A[\pi]$ if and only if there exists an isomorphism $\gamma: A^t\to A$ such that $\gamma\lambda=\pi$.
We further identify $M^t$ with $M^\vee$ and apply the equivalence functor $\Hom(E, -): \calA_E\to \mathrm{Lat}_R$ to transform the desired commutative diagram  on the left hand side of \eqref{eqn:cds} to a  commutative diagram of $R$-lattices on the right hand side:
\begin{equation}\label{eqn:cds}
\begin{tikzcd}
    A \ar[rd, "\lambda"']\ar[rr, "\pi"] & & A\\
    & A^t\ar[ru, "\exists \gamma"', "\simeq", dashrightarrow] & 
\end{tikzcd}\qquad \qquad
\begin{tikzcd}
    M\ar[r, "\pi"]\ar[d, "\lambda_*"'] & M\\
    M^t\ar[r, equal] & M^\vee \ar[u, dashrightarrow, "\exists \gamma_*"', "\simeq"]
\end{tikzcd}    
\end{equation}
Since the composition $M\xrightarrow{\lambda_*}M^t=M^\vee$ is just the canonical inclusion, there exists an isomorphism $\gamma_*: M^\vee\to M$ such that the diagram on the right hand side of \eqref{eqn:cds} commutes if and only if $\pi M^\vee=M$.
\end{proof}

\begin{prop}\label{prop:wtSigma}
    The functor $\Hom(E,-):\calA_{E}^{\rm pol}\to {\rm Lat}_R^H$ gives rise to a bijection between the set $\wt \Sigma_n(\sqrt{-p})$ and the set of isomorphism classes of positive-definite hermitian $\Z[\sqrt{-p}]$-lattices of rank $n$ such that $\sqrt{-p} M^\vee=M$. Morevoer, any two members in $\wt \Sigma_n(\sqrt{-p})$ belong to the same genus if and only if their associated positive-definite hermitian $R$-lattices belong to the same genus. 
\end{prop}

\section{$\sqrt{-p}$-modular hermitian $O_K$-lattices}\label{sec:3}

In this section, we fix a prime $p\in \bbN$ and write $K$ for the
imaginary quadratic field $\Q(\sqrt{-p})$.  Let $\sigma$ be the unique nontrivial element of $\Gal(K/\Q)$, which sends each $a\in K$ to its complex conjugate $a^\sigma$.   Let $V$ be a nonzero
$K$-vector space of finite dimension $n\coloneqq \dim V$.  Suppose that $V$
is equipped with a positive-definite hermitian form
$\varphi: V\times V\to K$ that is linear in the first variable
and anti-linear in the second variable.    

\begin{defn}
The \emph{determinant}
$d\varphi$ of the hermitian space $(V, \varphi)$ is defined to be the element of
$\Q^\times/\Nm_{K/\Q}(K^\times)$ represented by the determinant of the
Gram matrix $[\varphi(e_i, e_j)]_{1\leq i,j\leq n}$ for a $K$-basis
$\{e_i\}_{1\leq i \leq n}$ of $V$, which is 
 independent of the choice of the
basis. Here $\Nm_{K/\Q}: K\to \Q$ denotes the norm map, which is 
abbreviated as $\Nm:K\to \Q$.  
\end{defn}
By the positive-definiteness
assumption on $\varphi$, we have
$d\varphi\in \Q_{+}^\times/\Nm(K^\times)$, where $\Q_{+}^\times$
denotes the multiplicative group of positive rational numbers. Given $a\in \Q^\times_+$, we write $[a]$ for its class $a\Nm(K^\times)\in \Q_{+}^\times/\Nm(K^\times)$.
From
\cite[Theorem~2.2]{Shimura-HermForm-2008} or
\cite[Examples~10.1.6(iv)]{Scharlau-Quad-Herm}, the positive-definite
hermitian space $(V, \varphi)$ is uniquely determined up to isomorphism by its dimension
$n$ and determinant $d\varphi\in \Q_{+}^\times/\Nm(K^\times)$.

Let $O_K$ be the ring of integers of
$K$.  By \emph{an $O_K$-lattice in $V$}, we mean exclusively an $O_K$-lattice of \emph{full rank}, that is,  a finitely generated $O_K$-module in $V$ that contains a $K$-basis of $V$.  Similarly, we define the notion of $\Z[\sqrt{-p}]$-lattice in $V$.  Clearly, each $O_K$-lattice is automatically a $\Z[\sqrt{-p}]$-lattice. In light of Proposition~\ref{prop:wtSigma}, we are interested in the $\Z[\sqrt{-p}]$-lattices $L$ with $\sqrt{-p}L^\vee=L$, where $L^\vee\coloneqq\{\, x\in V\, | \, \varphi(x,L)\subseteq \Z[\sqrt{-p}]\, \}$ is the $\Z[\sqrt{-p}]$-valued dual lattice defined in \eqref{eqn:1.3}.
Observe that $\Z[\sqrt{-p}]$ coincides with $O_K$ if $p\not\equiv 3\pmod{4}$, and $\Z[\sqrt{-p}]$ is a suborder of index $2$ in $O_K$ if $p\equiv 3\pmod{4}$. For simplicity, we focus on $O_K$-lattices in this section and leave the more difficult $\Z[\sqrt{-p}]$-lattices for which  $\Z[\sqrt{-p}]\subsetneq O_K$ to section~\ref{sec:4}. For an $O_K$-lattice in $(V, \varphi)$, it is conceptually simpler to consider the $O_K$-valued dual lattice than the $\Z[\sqrt{-p}]$-valued dual lattice.


\begin{defn}\label{defn:mod-OK-lattice}
Given an $O_K$-lattice $L$ in $V$, we
define its \emph{$O_K$-valued  dual lattice} $L^{\star, \varphi}$ with respect to $\varphi$
as
\begin{equation}
  \label{eq:4}
L^{\star, \varphi}\coloneqq \{x\in V\mid \varphi(x, L)\subseteq O_K\}.  
\end{equation}
If $L^{\star, \varphi}=\gra^{-1}L$ for some fractional $O_K$-ideal
$\gra\subseteq K$, then we say that $L$ is
\emph{$\gra$-modular}.   In particular, we say that $L$ is \emph{$\sqrt{-p}$-modular} if $L=\sqrt{-p}L^{\star, \varphi}$. 
\end{defn}

If $p\not\equiv 3\pmod{4}$,  then $\Z[\sqrt{-p}]=O_K$  and hence $L^{\star, \varphi}=L^\vee$. For this case, the definition of  $\sqrt{-p}$-modular $O_K$-lattices is identical to the original one, which requires $\sqrt{-p}L^\vee=L$.   On the other hand, these two notions of $\sqrt{-p}$-modularity differ for the case $p\equiv 3\pmod{4}$, because 
\begin{equation}
  \label{eq:11}
  L^\vee=2L^{\star, \varphi}=L^{\star, \frac{1}{2}\varphi} \qquad \text{if}\quad p\equiv 3\pmod{4}.
\end{equation}
Indeed, in this case $2O_K$ is the maximal $O_K$-ideal contained in
$\Z[\sqrt{-p}]$. For every $x\in V$,  the set $\varphi(x, L)$ forms an $O_K$-module since $L$ is an $O_K$-lattice by assumption. Thus $L^\vee =\{x\in V\mid \varphi(x, L)\subseteq  2O_K\}$, from which \eqref{eq:11} follows directly. Nevertheless, this discrepancy cause no extra trouble because of the following lemma. 
\begin{lem}\label{lem:two-duals}
For $p\equiv 3\pmod{4}$, there is a bijection between the following set of $O_K$-lattices:
\begin{equation}\label{eqn:bij-3.3}
    \left\{\parbox{3.6cm}{$O_K$-lattices $M$ in $(V, \varphi)$ with $\sqrt{-p}M^\vee=M$
      }\right\}\longleftrightarrow
  \left\{\parbox{3.6cm}{$O_K$-lattices $L$ in $(V, \varphi)$ with $\sqrt{-p}L^{\star, \varphi}=L$}\right\}
\end{equation}
  that preserves isometric classes and genera. 
\end{lem}
\begin{proof}
 For each prime $\ell$, let $V_\ell\coloneqq V\otimes_\Q\Q_\ell$ be the $\ell$-adic completion of $V$, and $\varphi_\ell: V_\ell\times V_\ell\to K_\ell\coloneqq K\otimes_\Q\Q_\ell$ be the hermitian $K_\ell$-form on $V_\ell$ induced from $\varphi$. Since the order $\Z[\sqrt{-p}]$ is maximal at the odd prime $p$, for each $O_K$-lattice $L$ in $V$ we have 
 \[L^\vee\otimes\Z_p=\{x\in V_p\mid \varphi_p(x, L_p)\subseteq \Z_p[\sqrt{-p}]=O_{K_p}\}=L^{\star, \varphi}\otimes\Z_p.\]   
From Jacobowitz's classification \cite{Jacobowitz-HermForm}  of $\sqrt{-p}$-modular hermitian $O_{K_p}$-lattices to be recalled in Lemma~\ref{lem:even-rk}, both sides of \eqref{eqn:bij-3.3} are empty if $n$ is odd, in which case the lemma is vacuous. Now assume that $n$ is even.   From \eqref{eq:11}, an $O_K$-lattice $M$ in $(V, \varphi)$ satisfies
  $\sqrt{-p}M^\vee=M$ if and only if $M^{\star, \frac{1}{2}\varphi}=M$. Since $n$ is even, we have 
  $d(\frac{1}{2}\varphi)=d\varphi$ in
  $\Q_+^\times/\Nm(K^\times)$, so there exists an isometry 
  $\vartheta: (V, \frac{1}{2}\varphi)\to (V, \varphi)$ in this case. Therefore, an $O_K$-lattice $M\subset V$ belongs to the left hand side of \eqref{eqn:bij-3.3} if and only if its image $L\coloneqq \vartheta(M)$ belongs to the right hand side. This establishes the desired bijection, which preserves isometric classes and genera since $\vartheta$ is an isometry. 
\end{proof}



Lemma~\ref{lem:two-duals} allows us to treat $\sqrt{-p}$-modular $O_K$-lattices uniformly for all prime $p$ without worrying  about the distinction between $\Z[\sqrt{-p}]$-valued dual lattices and $O_K$-valued ones. For the rest of this section, we talk about $\sqrt{-p}$-modular $O_K$-lattices exclusively in the sense of Definition~\ref{defn:mod-OK-lattice}. 
If the hermitian form $\varphi$ is clear from the context, then we drop it from the superscript of 
 $L^{\star, \varphi}$ and simply write $L^\star$.
 

   \begin{defn}       
   The \emph{norm} $\grn(L)$ of a hermitian $O_K$-lattice $L$ in $(V, \varphi)$ is defined as the fractional $O_K$-ideal generated by the set $\{\varphi(x, x)\mid x\in L\}$.  At each prime $\ell$, the fractional $O_{K_\ell}$ norm ideal $\grn(L_\ell)$  is defined analogously. 
   \end{defn}
   The local and global norms fit together according to the rule 
   \begin{equation}\label{eqn:lg-norm}
       \grn(L)\otimes \Z_\ell=\grn(L_\ell),\qquad  \text{for every prime } \ell,  
   \end{equation}
from which it follows that $\grn(L)$ descends to be an invariant of the genus of $L$.

The main result of this section is as follows.

\begin{thm}\label{thm:main1}\begin{enumerate}[label={\upshape(\roman*)}, align=right, widest=iii,  leftmargin=*]
\item[{\rm (1)}] If there  exists a $\sqrt{-p}$-modular
   $O_K$-lattice $L$ in $(V, \varphi)$, then $n$ is even and $d\varphi=[1]$ in $\Q_{+}^\times/\Nm(K^\times)$.
\item[{\rm (2)}] Conversely, let $(V, \varphi)$ be an even dimensional hermitian $K$-space with $d\varphi=[1]$. There exists a $\sqrt{-p}$-modular
   $O_K$-lattice $L$ in $(V, \varphi)$ if and only if
   \begin{itemize}
       \item[{\rm (a)}] either $p\not\equiv 3\pmod{4}$, or
       \item[{\rm (b)}] $p\equiv 3\pmod{4}$ and $4\mid n$. 
   \end{itemize}
   \item[{\rm (3)}] Suppose that one of the conditions in \emph{(2)} holds so that there exists a $\sqrt{-p}$-modular $O_K$-lattice in $(V, \varphi)$. 
   \begin{itemize}
       \item[{\rm (a)}] If $p\not\equiv 3\pmod{4}$ and $4\mid n$, then there are two genera of $\sqrt{-p}$-modular $O_K$-lattices in $(V, \varphi)$, distinguished according to whether the norm of the lattice is equal to $pO_K$ or $2pO_K$. 
       \item[{\rm (b)}]  For the remaining cases,  there is a unique genus of such lattices,  for which the corresponding norm is $pO_K$. 
   \end{itemize}
   \end{enumerate}
\end{thm}



 Since the positive-definite
hermitian space $(V, \varphi)$ is uniquely determined by its dimension
$n$ and determinant $d\varphi$, 
 part (1) of Theorem~\ref{thm:main1} shows that for a fixed even dimension $n$,   there is a unique  hermitian space up to isomorphism that admits a $\sqrt{-p}$-modular $O_K$-lattice. 
To prove this theorem,  we first need a better understanding of the group $\Q_{+}^\times/\Nm(K^\times)$. 

 From the fundamental theorem of arithmetic,
$\Q_+^\times$ is the free abelian group generated by set of all
rational primes, which in turn implies that 
$\Q_{+}^\times/\Nm(K^\times)$ is an elementary 2-group
generated by the set of primes distinct from $p$.  As usual, we write
$\Lsymb{K}{\ell}$ for the Artin symbol \cite[p.~94]{vigneras} of $K$
at a rational prime $\ell$, which takes value $\{1, 0, -1\}$
respectively according
to whether $K$ is split, ramified or inert at $\ell$.  When $\ell$ is
odd, $\Lsymb{K}{\ell}$ coincides with the Legendre symbol
$\Lsymb{-p}{\ell}$. On the other hand, if $p$ is an odd prime, then 
\begin{equation}
  \label{eq:1}
\Lsymb{K}{2}=
\begin{dcases}
  1 &\text{if}\quad p\equiv 7\pmod{8};\\
  0 &\text{if}\quad p\equiv 1\pmod{4};\\
 -1 &\text{if}\quad p\equiv 3\pmod{8}.\\
\end{dcases}
\end{equation}
 For each prime $\ell\in \bbN$, let $(a, b)_\ell\in \{\pm 1\}$ be the
   Hilbert symbol of a pair of nonzero elements  $a, b\in \Q_\ell$,  which takes value $1$ if and only
  if $a\in \Nm(\Q_\ell(\sqrt{b})^\times)$.   
   In particular, we have an embedding of abelian groups 
   \[ \iota_\ell: \Q_\ell^\times/\Nm(K_\ell^\times) \embed \{\pm 1\}, \quad a\Nm(K_\ell^\times)\mapsto (a, -p)_\ell\]
To describe the structure of $\Q_{+}^\times/\Nm(K^\times)$, we
introduce two subgroups of it.
\begin{defn}
    Let $\Gin$ be the  subgroup of
  $\Q_{+}^\times/\Nm(K^\times)$ generated by the set of inert
  primes $\Pin\coloneqq \left\{ \text{prime } \ell\in \bbN
    \,\middle\vert\, \Lsymb{K}{\ell}=-1\right \}$, and 
  $\Grs$ be the subgroup generated by the remaining ramified or split primes.
\end{defn}
By construction, we have
  $\Q_{+}^\times/\Nm(K^\times)=\Gin\cdot \Grs$.

\begin{prop}\label{prop:determinant}
  \begin{enumerate}[label={\upshape(\roman*)}, align=right, widest=iii,  leftmargin=*]
\item[{\rm (1)}]  The elementary $2$-group $\Q_{+}^\times/\Nm(K^\times)$ is
  a direct product of $\Gin$ and $\Grs$, and $\Pin$ is an
  $\F_2$-basis of
  $\Gin$.
\item[{\rm (2)}] If either $p=2$ or  $p\equiv 3\pmod{4}$, then $\Grs$ is trivial. 
\item[{\rm (3)}] If $p\equiv 1\pmod{4}$, then $\Grs$ is a cyclic group of
  order $2$ whose generator $g\Nm(K^\times)$ 
is uniquely determined by the  following conditions \begin{equation}
  \label{eq:9}
   (g, -p)_2=-1, \qquad   (g, -p)_p=-1, \quad
  \text{and}\quad (g, -p)_\ell=1 \quad\forall  \ell\not\in \{2, p\}.
\end{equation}
  \end{enumerate}
\end{prop}
\begin{ex}
Before proving Proposition~\ref{prop:determinant}, we provide some concrete
examples of positive integers $g$ that represent the nontrivial element of
$\Grs$ when $p\equiv 1\pmod{4}$.  Let $\ell_0\in \bbN$ be a prime
that satisfies the
  following   conditions
 \begin{equation}
     \label{eq:2}
     \ell_0\equiv 3\pmod{4}\qquad \text{and}\qquad
     \Lsymb{\ell_0}{p}=-1. 
   \end{equation}
   A direct computation using
   \cite[Theorem~III.1]{Serre-arithmetic}  shows that $g=\ell_0$ has the desired Hilbert symbols in
  \eqref{eq:9}. By definition,  $2\Nm(K^\times)$ lies in $\Grs$  when
  $p\equiv 1\pmod{4}$, and a similar computation shows that it generates $\Grs$ if and only if  $p\equiv
  5\pmod{8}$.
\end{ex}

\begin{proof}[Proof of Proposition~\ref{prop:determinant}]
(1) From the Hasse
  norm 
  principle and \cite[Theorem~III.4]{Serre-arithmetic}, there is a short exact
  sequence
  \begin{equation}
    \label{eq:3}
    1\longrightarrow \Q_+^\times/\Nm(K^\times)\longrightarrow \bigoplus_{\text{prime
      }\ell}\Q_\ell^\times/\Nm(K_\ell^\times)\xrightarrow{
    \prod \iota_\ell} \{\pm 1\}\longrightarrow 1. 
\end{equation}
Here we omit the infinite place since $(a,
    -p)_\infty=1$ for every  positive rational number $a\in\Q_+^\times$.
   If $\ell$ is a
    prime split in $K$, then $K_\ell\simeq \Q_\ell\oplus\Q_\ell$, and 
    hence $\Nm(K_\ell^\times)=\Q_\ell^\times$. Thus 
    \eqref{eq:3} simplifies into
    \begin{equation}
      \label{eq:6}
    1\longrightarrow \Q_+^\times/\Nm(K^\times)\longrightarrow \bigoplus_{\Lsymb{K}{\ell}\in
      \{-1, 0\}}
      \Q_\ell^\times/\Nm(K_\ell^\times)\xrightarrow{\prod \iota_\ell} \{\pm 1\}\longrightarrow 1.        
    \end{equation}
    Here every $\Q_\ell^\times/\Nm(K_\ell^\times)$ with $\Lsymb{K}{\ell}\in \{-1, 0\}$ can be  canonically
    identified with the group $\{\pm 1\}$ by sending each coset $a\Nm(K_\ell^\times)$ to its
    corresponding     Hilbert symbol $(a, -p)_\ell$. More
    explicitly, 
    if $\ell$ is inert in $K$, then
    $\Q_\ell^\times/\Nm(K_\ell^\times)\simeq \ell^\Z/\ell^{2\Z}$, and
    $(a, -p)_\ell=(-1)^{\nu_\ell(a)}$, where $\nu_\ell:
    \Q_\ell^\times\twoheadrightarrow \Z$ denotes the normalized
    discrete valuation of $\Q_\ell$. Thus if $a\in \Z_{>0}$ is a square-free
positive   integer divisible by some $\ell$ inert in $K$, then $a$ is
not a local norm at $\ell$, and hence 
$a\not\in \Nm(K^\times)$. This proves part (1)  of the
proposition. 
  
(2) Next, we determine the group structure of $\Grs$. By the above
discussion, for each $a\in \Q_+^\times$, we have 
\begin{equation}
  \label{eq:10}
  \begin{split}
  a\Nm(K^\times)\in \Grs\qquad &                                         
                                            \Longleftrightarrow \quad
                                            a\Nm(K_\ell^\times)=\Nm(K_\ell^\times)\quad 
                                            \text{for every $\ell$
                                            inert in $K$}, \\
&                                         
                                            \Longleftrightarrow \quad    \qquad (a,
  -p)_\ell=1\quad \text{for every $\ell$ inert in $K$.}        
  \end{split}
\end{equation}
  Thus we can restrict \eqref{eq:6} to $\Grs$ to obtain a simpler short exact sequence
\begin{equation}
  \label{eq:8}
1\longrightarrow \Grs\longrightarrow \bigoplus_{\Lsymb{K}{\ell}=0}
  \Q_\ell^\times/\Nm(K_\ell^\times)\longrightarrow \{\pm 1\}\longrightarrow 1.   
\end{equation}
If either $p=2$ or $p\equiv 3\pmod{4}$, then the only prime ramified in
$K=\Q(\sqrt{-p})$  is $p$ itself. Hence  $ \Grs$ is trivial  by
the exactness of \eqref{eq:8}, proving part (2) of the proposition.

(3) Lastly, suppose that $p\equiv
1\pmod{4}$. Then there are exactly 2 primes ramified in $K$, namely,
$2$ and $p$. From \eqref{eq:8} again, we see 
that $\Grs$ is a cyclic group of order $2$ whose generator
$g\Nm(K^\times)$ is uniquely determined by the conditions in
\eqref{eq:9}. 
\end{proof}

 We return to the study of $\sqrt{-p}$-modular lattices in the
 hermitian space $(V, \varphi)$.  From the definition, the $\sqrt{-p}$-modularity of an $O_K$-lattice can be characterized  locally.  
 
 \begin{lem}
 An $O_K$-lattice $L$ in $(V, \varphi)$ is
 $\sqrt{-p}$-modular if and only if both of the following conditions hold
 \begin{enumerate}
 \item[{\rm (1)}]  $L_\ell$ is a
   self-dual $O_{K_\ell}$-lattice in $(V_\ell, \varphi_\ell)$ for every prime $\ell\neq p$;
 \item[{\rm (2)}] $L_p$ is a $\sqrt{-p}$-modular    $O_{K_p}$-lattice in
   $(V_p, \varphi_p)$. 
 \end{enumerate}    
 \end{lem}
 At each prime $\ell$ (including $\ell=p$), the existence of local
 $O_{K_\ell}$-lattices as above puts certain constraints on the local
 determinant $d\varphi_\ell$ that we shall describe below. 
 Recall from
\cite[Theorem~3.1]{Jacobowitz-HermForm} that a nondegenerate
hermitian space $(W_\ell, \psi_\ell)$ over the quadratic semisimple
$\Q_\ell$-algebra $K_\ell$ is uniquely
determined up to isometry by its rank and determinant.

 \begin{lem}\label{lem:self-dual-OKl-lattice}
   Let $(W_\ell, \psi_\ell)$ be a nondegenerate hermitian space
   over $K_\ell$.
   \begin{enumerate}[{\rm (1)}]
   \item If $\Lsymb{K}{\ell}\in \{0, 1\}$, then there is always a
     self-dual $O_{K_\ell}$-lattice in $(W_\ell, \psi_\ell)$;
    \item[{\rm (2)}]  if $\Lsymb{K}{\ell}=-1$, then there is  a
     self-dual $O_{K_\ell}$-lattice in $(W_\ell, \psi_\ell)$ if and
     only if $d\psi_\ell=1$ in
$\Q_\ell^\times/\Nm(K_\ell^\times)$.  
   \end{enumerate}
Moreover, if $\ell$ is unramified in $K$, then there is at most one isometric class of
self-dual $O_{K_\ell}$-lattices in $(W_\ell, \psi_\ell)$.  
 \end{lem}
 \begin{proof}
First, suppose that either $\Lsymb{K}{\ell}\in \{0, 1\}$ or
$d\psi_\ell=1$ in $\Q_\ell^\times/\Nm(K_\ell^\times)$. We claim that in each of these cases $d\psi_\ell$ is
represented by some unit $u\in \Z_\ell^\times$. If
$\Lsymb{K}{\ell}=1$, this follows directly from the triviality of
$\Q_\ell^\times/\Nm(K_\ell^\times)$. If $\Lsymb{K}{\ell}=0$, then the
canonical inclusion
$\Z_\ell^\times/\Nm(O_{K_\ell}^\times)\to
\Q_\ell^\times/\Nm(K_\ell^\times)$ is an isomorphism, so $d\psi_\ell$
is indeed represented by a unit.  This verifies the claim and allows us
to identify $(W_\ell,\psi_\ell)$ with the hermitian space
$(K_\ell^{\rk(W_\ell)}, \langle~,~\rangle)$ whose Gram matrix with
respect to the standard basis is given by the diagonal matrix
$\diag(1, \cdots, 1, u)$.  Then the
standard $O_{K_\ell}$-lattice $O_{K_\ell}^{\rk(W_\ell)}$ is clearly
self-dual.   From \cite[\S7]{Jacobowitz-HermForm}, if $\ell$ is inert in $K$,  then every self-dual $O_{K_\ell}$-lattice admits an orthonormal basis, so such lattices in $(W_\ell, \psi_\ell)$ form a single isometric class. We leave it as an easy exercise to show that there is a unique isometric class of self-dual $O_{K_\ell}$-lattices in $(W_\ell, \psi_\ell)$ in the split case  as well.

 
Conversely, suppose that there exists  a
  self-dual $O_{K_\ell}$-lattice  $L_\ell$ in $(W_\ell, \psi_\ell)$.  Then the Gram matrix
  $[\psi_\ell(e_i, e_j)]_{1\leq i, j\leq n}$ belongs to
  $\GL_n(O_{K_\ell})$ for any $O_{K_\ell}$-basis
  $\{e_1, \cdots, e_n\}$ of $L_\ell$.  In particular, 
  the  determinant $d\psi_\ell\in \Q_\ell^\times/\Nm(K_\ell^\times)$ is represented by a unit
  $\det[\psi_\ell(e_i, e_j)]\in \Z_\ell^\times$. If $\ell$ is further assumed to
  be inert in $K$, then
 $(d\psi_\ell, -p)_\ell=(-1)^{\nu_\ell(d\psi_\ell)}=1$, or equivalently, $d\psi_\ell=1$ in
$\Q_\ell^\times/\Nm(K_\ell^\times)$. This proves the necessity part of the second statement of the lemma. 
\end{proof}


Note that the extension $K/\Q$ admits a ramified prime $\ell\neq p$  if and only if $p\equiv 1\pmod{4}$, in which case $\ell=2$. The classification of self-dual lattices in the ramified dyadic case is more involved and will be postponed to Lemma~\ref{lem:RU-self-dual-genus}.  
Combining Lemma~\ref{lem:self-dual-OKl-lattice} and \eqref{eq:10}, we
immediately obtain the following corollary.  
\begin{cor}\label{cor:det-away-from-p}
If there exists a $\sqrt{-p}$-modular $O_K$-lattices in 
 $(V, \varphi)$, then $d\varphi\in \Grs$. 
\end{cor}



Now we fix a prime $\ell$  ramified in $K/\Q$ (e.g.~$\ell=p$) and consider $\sqrt{-p}$-modular hermitian $O_{K_\ell}$-lattices at the prime $\ell$. 
For simplicity, a hermitian $O_{K_\ell}$-lattice $(M_\ell, \langle~,~\rangle_\ell)$ is often abbreviated as $M_\ell$ since the hermitian form $\langle~,~\rangle_\ell: M_\ell\times M_\ell\to O_{K_\ell}$ will be implicitly present. 

\begin{defn}

The determinant $dM_\ell$ of $M_\ell$ is defined to be the element $\Q_\ell^\times/\Nm(O_{K_\ell}^\times)$ represented by the determinant of the
Gram matrix $[\langle e_i, e_j\rangle_\ell]_{i,j}$ for any $O_{K_\ell}$-basis $\{e_i\}_{1\leq i \leq \rk M_\ell}$ of $M_\ell$. 
The \emph{scale}  of $M_\ell$ is defined as the fractional $O_{K_\ell}$-ideal $\grs(M_\ell)\coloneqq \{\langle x, y\rangle_\ell\mid x, y\in M_\ell\}$, and the \emph{norm} $\grn(M_\ell)$ as the fractional $O_{K_\ell}$-ideal generated by the set $\{\langle x, x\rangle_\ell\mid x\in M_\ell\}$. 
Clearly, $\grn(M_\ell)\subseteq \grs(M_\ell)$, and we call $M_\ell$ \emph{normal} if the equality $\grn(M_\ell)=\grs(M_\ell)$ holds, and  \emph{subnormal} otherwise. 
\end{defn}

Hermitian lattices of rank $1$ and $2$ will be called \emph{lines} and \emph{planes}, respectively. 
They play crucial roles in the classification theory since every hermitian lattice can be written as an orthogonal sum of modular lines and planes by \cite[Proposition~4.3]{Jacobowitz-HermForm}. By definition, a line is normal.  Following Jacobowitz \cite[\S4]{Jacobowitz-HermForm},  if 
$M_\ell$ is a hermitian $O_{K_\ell}$-lattice with an orthogonal basis $e_1, \cdots, e_n$ for which $\langle e_i, e_i\rangle_\ell=a_i\in \Q_\ell$, then we write 
\[M_\ell\simeq (a_1)\perp \cdots \perp (a_n). \] 
Similarly, if $M_\ell=O_{K_\ell}e_1\oplus O_{K_\ell}e_2$ is a plane for which $\langle e_1, e_1\rangle_\ell=a, \langle e_1, e_2\rangle_\ell=\alpha$, and $\langle e_2, e_2\rangle_\ell=b$, then we write $M_\ell\simeq \begin{bmatrix}
    a & \alpha \\ \alpha^\sigma & b
\end{bmatrix}$.
For example, we define 
the
\emph{$\sqrt{-p}$-modular hyperbolic $O_{K_p}$-plane} $H_p(1)$ as  the hermitian $O_{K_p}$-plane  
\[ H_p(1)\simeq \begin{bmatrix}
  0 & \sqrt{-p} \\ -\sqrt{-p} & 0
\end{bmatrix}. \]  As the name suggests, $H_p(1)$ can easily be shown to be $\sqrt{-p}$-modular.  When $p\equiv 1\pmod{4}$ and $\ell=2$,  we also define the \emph{self-dual hyperbolic $O_{K_2}$-plane} $H_2(0)$ as the 
hermitian $O_{K_2}$-plane  $H_2(0)\simeq \begin{bmatrix}
  0 & 1 \\ 1 & 0
\end{bmatrix}$.

\begin{lem}\label{lem:even-rk}
Every $\sqrt{-p}$-modular hermitian $O_{K_p}$-lattice has even rank. 
\end{lem}
\begin{proof}
Let $M_p$ be a $\sqrt{-p}$-modular hermitian $O_{K_p}$-lattice. Since $M_p=\sqrt{-p}M_p^\star$, we have $\grs(M_p)=\sqrt{-p}O_{K_p}$, which coincides with the unique prime ideal of $O_{K_p}$. On the other hand,  the norm $\grn(M_p)$ is an even power of $\sqrt{-p}O_{K_p}$ because $K_p/\Q_p$ is ramified. Thus every $\sqrt{-p}$-modular $O_{K_p}$-lattice is subnormal. 
If $\rk M_p$ is odd, then $M_p$ is an orthogonal direct sum of $\sqrt{-p}$-modular lines by \cite[Proposition~4.4]{Jacobowitz-HermForm}, which is absurd since every line is normal. 
\end{proof}

Henceforth we focus exclusively on even dimensional nondegenerate hermitian spaces. 
When $p$ is odd, $K_p/\Q_p$ is a nondyadic ramified extension. As shown by Jacobowitz \cite{Jacobowitz-HermForm},  the classification of $\sqrt{-p}$-modular $O_{K_p}$-lattices in this case is considerably simpler than the dyadic case.


\begin{lem}\label{lem:det-at-p}
Let $p$ be an odd prime, and $(W_p, \psi_p)$ be an even dimensional nondegenerate hermitian $K_p$-space.  Then there exists a $\sqrt{-p}$-modular $O_{K_p}$-lattice in $(W_p, \psi_p)$ if and only if $d\psi_p=(-1)^{(\dim W_p)/2}$  in $\Q_p^\times/\Nm(K_p^\times)$, in which case there is a unique isometric class of 
such lattices in $(W_p, \psi_p)$.  Moreover, every $\sqrt{-p}$-modular $O_{K_p}$-lattice has norm $pO_{K_p}$.
\end{lem}

\begin{proof}
 From \cite[Proposition~8.1]{Jacobowitz-HermForm}, every
$\sqrt{-p}$-modular hermitian $O_{K_p}$-lattice  
is  isometric to $H_p(1)^{\oplus
  m}$ for some $m>0$.  Thus there exists  a $\sqrt{-p}$-modular $O_{K_p}$-lattice in $(W_p, \psi_p)$ if and only if $(W_p, \psi_p)\simeq H_p(1)^{\oplus (\dim W_p)/2}\otimes \Q_p$, or equivalently, 
\[ d\psi_p=\begin{vmatrix}
  0 & \sqrt{-p} \\ -\sqrt{-p} & 0
\end{vmatrix}^{(\dim W_p)/2}=(-p)^{(\dim W_p)/2}=(-1)^{(\dim W_p)/2}
  \in \Q_p^\times/\Nm(K_p^\times). \]
Clearly,  all such  lattices in $(W_p, \psi_p)$ form a single isometric class if they exist.   We compute directly that 
$\grn(H_p(1)^m)=\grn(H_p(1))=pO_{K_p}$. 
\end{proof}

We move on to the ramified dyadic case, which is further divided into two sub-cases.  If $p=2$, the dyadic extension $K_2/\Q_2$ belongs to the \emph{ramified prime} (``R-P'') type as named by Jacobowitz \cite{Jacobowitz-HermForm}, as $\Q_2(\sqrt{-2})/\Q_2$ is generated by  the square  root of a prime element of $\Q_2$. If $p\equiv 1\pmod{4}$,
 $\Q_2(\sqrt{-p})/\Q_2$ belongs to the \emph{ramified unit} (``R-U'') type, as  $-p$ is a unit in $\Z_2$. In both cases, we have $(-1, -p)_2=-1$, and hence the inclusion $\{\pm 1\}\subseteq\Z_2^\times$ induces the following canonical isomorphisms 
\begin{equation}\label{eqn:dyadic-normgp}
   \{\pm 1\}\simeq \Z_2^\times/\Nm(O_{K_2}^\times)\simeq \Q_2^\times/\Nm(K_2^\times), \quad\text{and}\quad \Q_2^\times/\Nm(O_{K_2}^\times)\simeq \{\pm 2^\Z\}.
\end{equation}

\begin{lem}\label{lem:RP-case}
   Let $p=2$, and $(W_2, \psi_2)$ be an even dimensional nondegenerate hermitian $K_2$-space. Then there always exists a $\sqrt{-2}$-modular $O_{K_2}$-lattice in $(W_2, \psi_2)$, and every such lattice  is isometric to  $M_2\perp H_2(1)^{(\dim W_2)/2-1}$ for some $\sqrt{-2}$-modular plane $M_2$.  More explicitly, we have
   \begin{itemize}
       \item $M_2\simeq \begin{bmatrix}
           -2 & \sqrt{-2} \\ -\sqrt{-2} & 4
       \end{bmatrix}$ if $d\psi_2=(-1)^{(\dim W_2)/2-1} \in \Q_2^\times/\Nm(K_2^\times) $; and 
       \item either $M_2= H_2(1)$ or $M_2\simeq \begin{bmatrix}
           -2 & \sqrt{-2} \\ -\sqrt{-2} & 0
       \end{bmatrix}$ if $d\psi_2=(-1)^{(\dim W_2)/2} \in \Q_2^\times/\Nm(K_2^\times)$.
   \end{itemize}
   In particular, if  $d\psi_2=(-1)^{(\dim W_2)/2-1}$, then there is a unique isometric class of $\sqrt{-2}$-modular $O_{K_2}$-lattice in $(W_2, \psi_2)$. If $d\psi_2=(-1)^{(\dim W_2)/2}$, then there are two isometric classes of such lattices distinguished according to whether the norm is equal to $4O_{K_2}$ or $2O_{K_2}$.
\end{lem}
\begin{proof}
From \cite[Proposition~10.3]{Jacobowitz-HermForm}, every $\sqrt{-2}$-modular $O_{K_2}$-lattice $L_2$ is isometric to $M_2\perp H_2(1)^m$ for some $m\geq 0$, where $M_2$ is  a $\sqrt{-2}$-modular plane with $\grn(M_2)=\grn(L_2)$.  Thus there exists a $\sqrt{-2}$-modular lattice in $(W_2, \psi_2)$ if and only if 
$(W_2, \psi_2)$ is isometric to $(M_2\perp H_2(1)^{(\dim W_2)/2-1})\otimes \Q_2$ for for some  $\sqrt{-2}$-modular plane $M_2$, or equivalently
\begin{equation}\label{eqn:RP-det}
   d\psi_2=dM_2\cdot (-1)^{\frac{\dim W_2}{2}-1}\in \Q_2^\times/\Nm(K_2^\times)\qquad \text{for some } M_2.  
\end{equation}
It remains to classify all $\sqrt{-2}$-modular $O_{K_2}$-planes and show that an $M_2$ satisfying \eqref{eqn:RP-det} always exists. From \cite[Proposition~10.4]{Jacobowitz-HermForm}, a $\sqrt{-2}$-modular plane $M_2$ is uniquely determined up to isomorphism by its norm $\grn(M_2)$ and determinant $dM_2\in \Q_2^\times/\Nm(O_{K_2}^\times)$. 
Combining \cite[ Proposition~9.1.a) and (9.1)]{Jacobowitz-HermForm}, we get
\begin{equation}\label{eqn:norm-bound-RP}
    \sqrt{-2}O_{K_2}=\grs(M_2)\supsetneqq \grn(M_2)\supseteq \grn(H_2(1))=4O_{K_2},
\end{equation}
where the equality $\grn(M_2)=\grn(H_2(1))$ holds if and only if $M_2\simeq H_2(1)$ by \cite[Proposition~9.2]{Jacobowitz-HermForm}. If $\grn(M_2)\supsetneqq\grn(H_2(1))$, then  $\grn(M_2)=2O_{K_2}$ by \eqref{eqn:norm-bound-RP}. On the other hand, since  $M_2$ is  a $\sqrt{-2}$-modular plane, its determinant $dM_2$ necessarily generates the ideal $2O_{K_2}$, so  $dM_2\in \{\pm 2\}\subseteq\Q_2^\times/\Nm(O_{K_2}^\times)$. A direct application of  \cite[Proposition~10.2.b)]{Jacobowitz-HermForm} implies that when   $\grn(M_2)=2O_{K_2}$,  we have  
\begin{equation}
    M_2\simeq \begin{dcases}
        \begin{bmatrix}
           -2 & \sqrt{-2} \\ -\sqrt{-2} & 4
       \end{bmatrix} &\text{if } dM_2=2; \\
       \begin{bmatrix}
           -2 & \sqrt{-2} \\ -\sqrt{-2} & 0
       \end{bmatrix} &\text{if } dM_2=-2. 
    \end{dcases} 
\end{equation}
The  lemma now follows immediately from the above classification of $M_2$. 
\end{proof}

 Suppose that $p\equiv 1\pmod{4}$ so that $\ell=2$ is a ramified prime in $K/\Q$ distinct from $p$.  We have seen in Lemma~\ref{lem:self-dual-OKl-lattice} that self-dual $O_{K_2}$-lattices exist in every non-degenerate hermitian $K_2$-space. It remains to classify them into isometric classes.   
 

\begin{lem}\label{lem:RU-self-dual-genus}
   Suppose that $p\equiv 1\pmod{4}$.  Let $(W_2, \psi_2)$ be an even dimensional nondegenerate hermitian $K_2$-space, and $L_2$ be a self-dual $O_{K_2}$-lattice in $(W_2, \psi_2)$.  
   \begin{enumerate}[(1)]
       \item[{\rm (1)}] If $d\psi_2=(-1)^{(\dim W_2)/2-1} \in \Q_2^\times/\Nm(K_2^\times)$, then $L_2$ is a normal lattice isometric to \[ (1)\perp (1)\perp H_2(0)^{(\dim W_2)/2-1}, \]
       and hence $\grn(L_2)=O_{K_2}$. 
       In particular, there is a unique isometric class of self-dual $O_{K_2}$-lattices in this case.  
       \item[{\rm (2)}] If $d\psi_2=(-1)^{(\dim W_2)/2} \in \Q_2^\times/\Nm(K_2^\times)$, then 
       \[L_2\simeq \begin{cases}
           (1)\perp (-1)\perp H_2(0)^{(\dim W_2)/2-1} & \text{if $L_2$ is normal};\\
            H_2(0)^{(\dim W_2)/2} & \text{if $L_2$ is subnormal}.
       \end{cases}\]
       In particular, 
       there are two isometric classes of self-dual $O_{K_2}$-lattices, distinguished by whether the lattice is normal or subnormal (or equivalently, by the norm of the lattice). If $L_2$ is normal, then $\grn(L_2)=O_{K_2}$; if $L_2$ is subnormal, then $\grn(L_2)=2O_{K_2}$.
   \end{enumerate}
\end{lem}
\begin{proof}
From \cite[Proposition~10.3]{Jacobowitz-HermForm} again, every self-dual $O_{K_2}$-lattice $L_2$ in $(W_2, \psi_2)$ is isometric to $M_2\perp H_2(0)^{(\dim W_2)/2-1}$ for some self-dual $O_{K_2}$-plane $M_2$  with \begin{equation}\label{eqn:113}
    \grn(M_2)=\grn(L_2)\qquad \text{and}\qquad dM_2 \cdot (-1)^{\frac{\dim W_2}{2}-1}=d\psi_2 \in \Q_2^\times/\Nm(K_2^\times).  
\end{equation} 
Since $M_2$ is self-dual, its determinant $dM_2$ lies in the group  $\Z_2^\times/\Nm(O_{K_2}^\times)$, which has been canonically identified with $\{\pm 1\}$ in \eqref{eqn:dyadic-normgp}. Similarly to the R-P case treated in Lemma~\ref{lem:RP-case},  from \cite[Proposition~10.4]{Jacobowitz-HermForm}   the self-dual plane $M_2$ is uniquely determined up to isomorphism by its norm $\grn(M_2)$ and determinant $dM_2\in \{\pm1\}$. 
Combining \cite[ Proposition~9.1.a) and (9.1)]{Jacobowitz-HermForm}, we get
\begin{equation}\label{eqn:norm-bound-RU}
    O_{K_2}=\grs(M_2)\supseteq \grn(M_2)\supseteq \grn(H_2(0))=2O_{K_2}.
\end{equation}
Now it follows from \cite[Propositions~9.2 and 9.3]{Jacobowitz-HermForm} and \eqref{eqn:norm-bound-RU} that the following are equivalent: 
\begin{itemize}
    \item $\grn(M_2)=\grn(H_2(0))$;
    \item $M_2$ (and hence $L_2$) is subnormal; 
    \item $M_2\simeq H_2(0)$, or equivalently,  $L_2\simeq H_2(0)^{(\dim W_2)/2}$.
\end{itemize}
Comparing the determinants, we see that there exists a self-dual $O_{K_2}$-lattice $L_2$ in $(W_2, \psi_2)$ with $L_2\simeq H_2(0)^{(\dim W_2)/2}$ if and only if $d\psi_2=(-1)^{(\dim W_2)/2}\in \Q_2^\times/\Nm(K_2^\times)$, in which case all subnormal self-dual $O_{K_2}$-lattices in $(W_2, \psi_2)$ form a single isometric class represented by such an $L_2$.

If $\grn(M_2)\supsetneqq \grn(H_2(0))$, then necessarily $M_2$ is normal with $\grn(M_2)=O_{K_2}$.  Combining \eqref{eqn:113} with \cite[Proposition~10.2.a)]{Jacobowitz-HermForm}, we find that 
\begin{equation}
    M_2\simeq \begin{dcases}
        (1)\perp (1) &\text{if } d\psi_2=(-1)^{(\dim W_2)/2-1}; \\
       (1)\perp (-1) &\text{if } d\psi_2=(-1)^{(\dim W_2)/2}. 
    \end{dcases} 
\end{equation}
Consequently, there always exists a unique isometric class of 
normal self-dual $O_{K_2}$-lattices in $(W_2, \psi_2)$ as described in the statement of the lemma.   
\end{proof}



\begin{proof}[Proof of Theorem~\ref{thm:main1}]
First, suppose there exists a $\sqrt{-p}$-modular
  $O_K$-lattice in the $n$-dimensional positive-definite hermitian $K$-space $(V, \varphi)$. We have seen in Lemma~\ref{lem:even-rk} and Corollary~\ref{cor:det-away-from-p} respectively that $2\mid n$  and $d\varphi\in \Grs$.  The latter already implies that $d\varphi=[1]$ in
  $\Q_+^\times/\Nm(K^\times)$ for the case $p\not\equiv  1\pmod{4}$  since $\Grs$ is trivial in this case by Proposition~\ref{prop:determinant}(2).  If $p\equiv 1\pmod{4}$, then $\Grs$ is a cyclic group of
  order $2$ whose generator $g\Nm(K^\times)$ is  characterized by \eqref{eq:9}. On the
  other hand,  locally at $p$ the existence of a $\sqrt{-p}$-modular $O_{K_p}$-lattice in $(V_p, \varphi_p)$ implies that $d\varphi_p=(-1)^{n/2}\in \Q_p^\times/\Nm(K_p^\times)$  by
  Lemma~\ref{lem:det-at-p}. The compatibility between the local and global determinants forces that 
  \begin{equation}\label{eqn:hbs-1}
     (d\varphi, -p)_p=(d\varphi_p, -p)_p=((-1)^{n/2}, -p)_p=\Lsymb{-1}{p}^{n/2}=1\neq (g, -p)_p,  
  \end{equation}
  Therefore, $d\varphi=[1]\in \Q_+^\times/\Nm(K^\times)$ for the case $p\equiv 1\pmod{4}$ as well. This proves part (1) of Theorem~\ref{thm:main1}.

Keep the existence assumption of $\sqrt{-p}$-modular $O_K$-lattices in $(V, \varphi)$. If $p\equiv 3\pmod{4}$,  a similar local-global compatibility consideration yields the desired condition $4\mid n$ on the dimension of $V$ as follows. Indeed, for this case we already know that  $d\varphi=[1]$ and $d\varphi_p=(-1)^{n/2}$. Putting them together, we get 
   \[1=(d\varphi, -p)_p=(d\varphi_p, -p)_p=((-1)^{n/2},
    -p)_p=\Lsymb{-1}{p}^{n/2}=(-1)^{n/2},\]
  which  implies that $4\mid n$.  This proves the necessity part of statement (2) of  Theorem~\ref{thm:main1}.

Conversely, suppose that $d\varphi=[1]$ and one of the following conditions
  holds:
  \begin{enumerate}
   \item[(a)] $p\equiv 3\pmod{4}$,  and $4\mid n$,  
   \item[(b)] $p\not\equiv 3\pmod{4}$,  and $2\mid n$.
  \end{enumerate}
  We prove that there exists a $\sqrt{-p}$-modular $O_K$-lattice $L$ in
  $(V, \varphi)$. Since $(V,\varphi)$ is positive definite with
  $d\varphi=[1]$, we may as well identify $(V,\varphi)$ with the
  standard hermitian space $(K^n, \langle~,~\rangle)$, whose Gram
  matrix with respect to the standard basis is given by the identity
  matrix.  Let $L_0\coloneqq O_K^n$ be the standard lattice in
  $(K^n, \langle~,~\rangle)$, which is clearly self-dual.  From Lemmas~\ref{lem:det-at-p} and~\ref{lem:RP-case},  there exists a $\sqrt{-p}$-modular
  $O_{K_p}$-lattice $N_p$ in  $(K_p^n, \langle~,~\rangle_p)$. 
   Now  from the local-global
  principle of lattices \cite[Proposition~4.21]{curtis-reiner:1}, 
  there exists an $O_K$-lattice $L$ in $K^n$ such that $L_p=N_p$ and
  $L_\ell=L_0\otimes \Z_\ell$ for every prime $\ell\neq p$. Such an
  $O_K$-lattice $L$ is $\sqrt{-p}$-modular by construction. This completes the proof of statement (2) of  Theorem~\ref{thm:main1}.

It remains to  classify the genera of $\sqrt{-p}$-modular $O_K$-lattices in $(V, \varphi)$. We keep the even dimensional hermitian space $(V, \varphi)= (K^n, \langle~,~\rangle)$ fixed as above and assume further that $4\mid n$ if $p\equiv 3\pmod{4}$. Let $L$ be a $\sqrt{-p}$-modular $O_K$-lattice in $(V, \varphi)$. If $\ell\in \bbN$ is a prime unramified in $K$, then there is a unique isometric class of self-dual $O_{K_\ell}$-lattices in $(V_\ell, \psi_\ell)$ by Lemma~\ref{lem:self-dual-OKl-lattice}. In particular, $L_\ell\simeq L_0\otimes Z_\ell$, where $L_0=O_K^n$ is the standard $O_K$-lattice as above. Thus the classification of genera is reduced to purely local classifications at the ramified primes.  Analogously, the computation of the norm $\grn(L)$ is reduced to local computations at the ramified primes since 
\[\grn(L)\otimes \Z_\ell=\grn(L_0\otimes Z_\ell)=O_{K_\ell} \qquad \text{for every unramified prime $\ell$}.\]

Suppose first that $p$ is odd.  From Lemma~\ref{lem:det-at-p}, there is a unique isometric class of $\sqrt{-p}$-modular $O_{K_p}$-lattices in $(V_p, \varphi_p)$, everyone of which has norm $pO_{K_p}$.  If we further assume that $p\equiv 3\pmod{4}$, then $p$ is the unique ramified prime for $K/\Q$. Thus in this case  there is a unique genus of $\sqrt{-p}$-modular $O_K$-lattices in $(V, \varphi)$, all of which have norm $pO_K$.  

The $p\equiv1\pmod{4}$ case is more involved, as $\ell=2$ is an extra prime ramified in $K$.  In light of above mentioned uniqueness result locally at $p$,   the number of genera of $\sqrt{-p}$-modular $O_K$-lattices in $(V, \varphi)$ coincides with the number of isometric classes of self-dual $O_{K_2}$-lattices in $(V_2, \varphi_2)$.
It has been shown in  Lemma~\ref{lem:RU-self-dual-genus} that the classification of self-dual lattices in $(V_2, \varphi_2)$ depends on the relationship between the determinant $d\varphi_2$ and the parity of $n/2$.  Since $d\varphi_2$ has been fixed to be $1$ by the assumption, our discussion is separated  into two sub-cases depending on the parity of $n/2$. If $n\equiv 2\pmod{4}$, then by  Lemma~\ref{lem:RU-self-dual-genus} there is a unique isometric class of self-dual $O_{K_2}$-lattices in $(V_2, \varphi_2)$,  all of which have norm $O_{K_2}$.  Therefore, in this case there is a unique genus of $\sqrt{-p}$-modular $O_K$-lattice in $(V, \varphi)$, all of which have norm $pO_K$. If $4\mid n$, then there are two isometric classes of self-dual $O_{K_2}$-lattices in $(V_2, \varphi_2)$, distinguished according to whether the corresponding norm is equal to $O_{K_2}$ or $2O_{K_2}$.  Consequently, when $p\equiv 1\pmod{4}$ and $4\mid n$,  there are two genera  of $\sqrt{-p}$-modular $O_K$-lattice in $(V, \varphi)$, distinguished according to whether the corresponding norm is equal to $pO_K$ or $2pO_K$. 

Lastly, suppose that $p=2$. In this case, $2$ is the only prime ramified in $K$.  
By Lemma~\ref{lem:RP-case}, the number of isometric classes of  $\sqrt{-2}$-modular $O_{K_2}$-lattices in $(V_2, \varphi_2)$ again depends on the parity of $n/2$ since $d\varphi_2=1$. If $n\equiv 2\pmod{4}$, then there is a unique isometric class of $\sqrt{-2}$-modular $O_{K_2}$-lattices in $(V_2, \varphi_2)$,  all of which have norm  $2O_{K_2}$.  If $4\mid n$, then there are two isometric classes of such lattices in $(V_2, \varphi_2)$ distinguished according to whether the corresponding norm is equal to  $2O_{K_2}$ or $4O_{K_2}$.   The global classification of the genera of $\sqrt{-p}$-modular $O_K$-lattices in $(V, \varphi)$ follows accordingly. 
\end{proof}





\section{$\sqrt{-p}$-modular hermitian $\Z[\sqrt{-p}]$-lattices when
  $p\equiv 3\pmod{4}$} \label{sec:4}

We keep the notation from the previous section and further assume that
$p\equiv 3\pmod{4}$ throughout this section, in particular, $p\neq 2$. The quadratic order
$R\coloneqq \Z[\sqrt{-p}]$ is a suborder of conductor $2$ in
$O_K\coloneqq \Z[\frac{-1+\sqrt{-p}}{2}]$.  Recall that an $R$-lattice $L$ in the positive definite hermitian $K$-space $(V, \varphi)$ is said to be \emph{$\sqrt{-p}$-modular} if $\sqrt{-p}L^\vee=L$, where     $L^\vee\coloneqq \{x\in V\mid \varphi(x, L)\subseteq R\}$ is the $R$-valued dual lattice of $L$.  
The goal of this section is
to give the necessary and  sufficient condition for the existence of
 $\sqrt{-p}$-modular $R$-lattices $L$ in $(V, \varphi)$ and to classify the genera of such lattices.

Firstly, an $O_K$-lattice is automatically an $R$-lattice. The following result follows directly from the combination of  Lemma~\ref{lem:two-duals} and Theorem~\ref{thm:main1}.
\begin{lem}\label{lem:modular-OK-lat-v2}
  There exists   an  $O_K$-lattice $L$ in $(V, \varphi)$ satisfying $\sqrt{-p}L^\vee=L$  if and only if $4\mid n$ and $d\varphi=[1]$, in which case there is a unique genus of such $O_K$-lattices. 
\end{lem}



We now move on to the study of general $\sqrt{-p}$-modular
$R$-lattices.    Since $R_\ell=O_{K_\ell}$ for every odd prime
$\ell$ (including $\ell=p$), an $R$-lattice $L$ satisfies
$L^\vee=\sqrt{-p}L$ if and only if all of the following hold:
\begin{enumerate}[(i)]
  \item $L_2$ is a self-dual $R_2$-lattice in $(V_2, \varphi_2)$;
  \item $L_\ell$ is a self-dual $O_{K_\ell}$-lattice in $(V_\ell,
    \varphi_\ell)$ for every odd prime $\ell\neq p$;
  \item $L_p$ is a $\sqrt{-p}$-modular $O_{K_p}$-lattice in $(V_p,
      \varphi_p)$. 
\end{enumerate}

Combining Lemmas~\ref{lem:self-dual-OKl-lattice} and
\ref{lem:det-at-p} together with the short exact sequence
\eqref{eq:6}, we obtain the following lemma. 

\begin{lem}\label{lem:nec-cond-mod-latt}
If there exists a $\sqrt{-p}$-modular $R$-lattice in 
$(V, \varphi)$, then necessarily one of the following holds:
\begin{enumerate}[(i)]
\item[{\rm (a)}] $p\equiv 7\pmod{8}$, $4\mid n$  and $d\varphi=[1]$.
\item[{\rm (b)}] $p\equiv 3\pmod{8}$, $4\mid n$  and $d\varphi=[1]$. 
\item[{\rm (c)}]   $p  \equiv 3\pmod{8}$, $ n\equiv 2\pmod{4}$  and
  $d\varphi=[2]$. 
\end{enumerate}
\end{lem}
\begin{proof}
  Since $p\equiv 3\pmod{4}$ by our assumption, we see from
  Proposition~\ref{prop:determinant} that an $\F_2$-basis of the
  elementary $2$-group $\Q_+^\times/\Nm(K^\times)$ is given by the set
  of inert primes
  $\Pin\coloneqq \left\{ \text{prime } \ell\in \bbN
    \,\middle\vert\, \Lsymb{K}{\ell}=-1\right \}$. 
  Suppose that there exists an $R$-lattice $L$  in $(V,
\varphi)$ satisfying $L^\vee=\sqrt{-p}L$. Then $L_\ell$ is  
a self-dual $O_{K_\ell}$-lattice in $(V_\ell,
    \varphi_\ell)$ for every odd prime $\ell\neq p$,  and $L_p$ is a $\sqrt{-p}$-modular $O_{K_p}$-lattice in $(V_p,
      \varphi_p)$.  
It follows from
Lemmas~\ref{lem:self-dual-OKl-lattice} and \ref{lem:det-at-p}
respectively  that
\begin{itemize}
\item $(d\varphi, -p)_\ell=1$
  for every odd prime $\ell\neq p$;
  \item  $n=\dim V$ is even and
$(d\varphi, -p)_p=(-1)^{n/2}$. 
\end{itemize}
  In particular, the $\ell$-adic valuation
$\nu_\ell(d\varphi)$ is even for every odd prime $\ell$ inert in
$K$. Thus $d\varphi$ lies in the subgroup of $\Q_+^\times/\Nm(K^\times)$ which is generated by the element
$2\Nm(K^\times)$, as $p$ is the only prime ramified in $K$ and $p\in \Nm(K^\times)$. 

First suppose that $p\equiv 7\pmod{8}$, then $2$ is split in $K$, and
$(2, -p)_2=(2, -p)_p=1$.  It follows from the short exact sequence
\eqref{eq:6} that $2\in \Nm(K^\times)$, so $d\varphi=[1]$ in this case. 
Comparing with the local condition $(d\varphi, -p)_p=(-1)^{n/2}$ at $p$, we conclude that $n$ is divisible
by $4$ in this case.

Next, suppose that $p\equiv 3\pmod{8}$, then $2$ is inert in $K$, and
$(2, -p)_2=(2, -p)_p=-1$.   Thus $2\Nm(K^\times)$ generates a cyclic
group of order $2$ in $\Q_+^\times/\Nm(K^\times)$, and this case is
further divided into two subcases depending on whether $d\varphi=[1]$ or
$d\varphi=[2]$.  In each subcase, the local equality  $(d\varphi,
-p)_p=(-1)^{n/2}$ translates into  the following condition on $n$ respectively:  
\[n\equiv
  \begin{cases}
    0 \pmod{4} &\text{if } d\varphi=[1];\\
        2 \pmod{4} &\text{if } d\varphi=[2]. 
  \end{cases}\qedhere
  \]
\end{proof}

The classification of self-dual hermitian $R_2$-lattices has been
carried out by the third named author, where
he establishes the following result in \cite[Theorems~1.1 and 1.2]{Yu-CF-PAMQ-2012}.

\begin{prop}\label{prop:yu-bij}
For every positive integer $n$, there is a bijection between the following two sets
\begin{equation}
  \label{eq:12}
\Latt_n\coloneqq \left\{\parbox{4.0cm}{isometric classes of
      self-dual hermitian $R_2$-lattices of
      $\Z_2$-rank
      $2n$}\right\}\longleftrightarrow
  \left\{\parbox{3.8cm}{isometric classes of symmetric
      $\F_2$-bilinear forms on the space
      $(\F_2)^n$}\right\} \eqqcolon \Bil_n.
\end{equation}   
Here all symmetric $\F_2$-bilinear forms on $\F_2^n$ are taken into
consideration, including the degenerate ones. 
\end{prop}
 Although the
statements of \cite[Theorems~1.1 and 1.2]{Yu-CF-PAMQ-2012} look almost
identical, the classification of the hermitian $R_2$-lattices there
needs to be separated into two different cases and treated
individually according to whether $2$ is split or inert in $K$. In this
section we provide an alternative approach that treats both cases
uniformly.

The starting point of our approach is the observation that $R_2$ is a Bass order as studied in the original work of Bass  \cite{Bass-1962} himself. The precise definition of Bass orders and the structural
theorem of lattices over commutative Bass orders given by Borevich and
Faddeev \cite{MR0205980}  will be recalled in the next section (cf.~Curtis-Reiner~\cite[Section 37,
p.~789]{curtis-reiner:1}).   When applied to the present case, this theorem merely
says that every $R_2$-lattice $N$ in a free $K_2$-module admits
a direct sum  decomposition 
\begin{equation}
  \label{eq:13}
  N=N_1\oplus N_2,
\end{equation}
where $N_1$ is a free $R_2$-lattice, and $N_2$ is a free
$O_{K_2}$-lattice.  While the decomposition \eqref{eq:13} is not
necessarily unique in general,  the ordered pair of non-negative integers $(r, s)\coloneqq
(\rk_{R_2}(N_1), \rk_{O_{K_2}}(N_2))$ is uniquely determined by $N$
itself and will be called \emph{the type} of $N$.

\begin{rem}\label{rem:type-cal}
    The type $(r,s)$ of an $R_2$-lattice $N$ can be calculated as follows. Let $\wt{N}$ be the $O_{K_2}$-module generated by
$N$ in $N\otimes_{R_2} K_2$. 
Then $\wt{N}$ is a free $O_{K_2}$-lattice  with
$2\wt{N}\subseteq N\subseteq \wt{N}$. Since $R_2/2O_{K_2}\simeq \F_2$
while $\dim_{\F_2}(O_{K_2}/2O_{K_2})=2$,  the ordered pair $(r, s)$
can be computed by the following equation
\begin{equation}
  \label{eq:14}
\left\{  \begin{aligned}
           r+s&=\rk_{O_{K_2}}(\wt{N})=\rk_{K_2}(N\otimes\Q_2), \\
           r+2s&=\dim_{\F_2} (N/2\wt{N}).
  \end{aligned}\right.
\end{equation}
\end{rem}

\begin{defn}
Denote by $\calL_0$ the \emph{unimodular $O_{K_2}$-line} $O_{K_2}e$ equipped with a hermitian pairing $\langle~,~\rangle_0$ such that $\langle e, e\rangle_0=2$.
    
\end{defn}

Suppose that $N$ is equipped with a \emph{perfect} hermitian pairing $\langle~,~\rangle: N\times N\to R_2$, so that $(N, \langle~,~\rangle)$ forms a unimodular hermitian $R_2$-lattice.  
For simplicity, we often drop $\langle~,~\rangle$ from the notation and simply say that $N$ is a unimodular hermitian $R_2$-lattice whenever the hermitian pairing $\langle~,~\rangle$ is clear from the context. 
In the next section, we develop an orthogonal decomposition theorem (namely, Theorem~\ref{thm:orthogonal-decomp}) for unimodular hermitian lattices over commutative local Bass orders.  In the present case, this theorem combined with the Witt cancellation law \cite[Propositions~4.4 and 5.7]{Yu-CF-PAMQ-2012} of the third named author yields the following proposition. 

\begin{prop}\label{prop:orth-decomp-R2}
Let  $N$ be a unimodular hermitian $R_2$-lattice of type $(r, s)$. Then there is an orthogonal decomposition $N=N_1\perp N_2$
such that 
\begin{enumerate}
    \item[{\rm (1)}] $N_1$ is a unimodular free $R_2$-lattice of rank $r$ uniquely determined up to isometry by $N$;
    \item[{\rm (2)}] $N_2$ is a unimodular free $O_{K_2}$-lattice isometric to $\calL_0^s$. 
\end{enumerate}
In particular, if we write $\Latt_n^r$ for the subset of $\Latt_n$ consisting of all isometric classes of unimodular hermitian $R_2$-lattices of type $(r, n-r)$, then the map $\Latt_n^r\to \Latt_r^r$ sending each $[N]\in \Latt_n^r$ to  $[N_1]\in \Latt_r^r$ establishes a well-defined bijection. 
\end{prop}
\begin{proof}
    From Theorem~\ref{thm:orthogonal-decomp}, there exists an 
orthogonal decomposition $N=N_1\perp N_2$ 
such that $N_1$ is a unimodular free $R_2$-lattice of rank $r$, and $N_2$ is a unimodular free $O_{K_2}$-lattice of rank $s$.  For each $i=1,2$, let $\langle~,~\rangle_i: N_i\times N_i\to R_2$ be the restriction of the   hermitian pairing $\langle~,~\rangle: N\times N\to R_2$ to $N_i$. The same argument as for \eqref{eq:11} shows that $\langle~,~\rangle_2$ takes values in $2O_{K_2}$, and $\frac{1}{2}\langle~,~\rangle_2: N_2\times N_2\to O_{K_2}$ defines a perfect hermitian pairing of $O_{K_2}$-lattices. 
Similarly,  $\calL_0$ is indeed a unimodular $R_2$-lattice because $\frac{1}{2}\langle~,~\rangle_0$ defines a perfect hermitian pairing on the $O_{K_2}$-line $O_{K_2}e$. 
Since $K/\Q$ is unramified at $2$, it follows from Lemma~\ref{lem:self-dual-OKl-lattice} that $N_2\simeq \calL_0^s$. Now we apply  the Witt cancellation law \cite[Propositions~4.4 and 5.7]{Yu-CF-PAMQ-2012}  to see that $N_1$ is uniquely determined up to isometry by $N$. 
\end{proof}

We should caution that the Witt cancellation law cited above depends crucially on the fact that $O_{K_2}$ is unramified over  $\Z_2$.  Such a result does not hold in general as demonstrated by Example~\ref{ex:no-witt}.  Proposition~\ref{prop:orth-decomp-R2} reduces the classification of unimodular hermitian $R_2$-lattices to the case where  the lattice in consideration is free over  $R_2$.

We  list some basic examples of unimodular hermitian free
$R_2$-lattices that will serve as building blocks of more general
ones. 

\begin{defn} \ 
    \begin{enumerate}        
        \item Denote by $\calL_1$ the \emph{unimodular $R_2$-line} $(R_2 e, \<~,~\>)$ with $\langle e, e\rangle=1$.
        \item Denote by $\calH$ the \emph{unimodular $R_2$-plane} $(R_2e_1\oplus R_2 e_2, \<~,~\>)$ 
  whose Gram matrix with respect to the standard basis $\{e_1, e_2\}$ is given by
  $\begin{bmatrix} 0 & 1 \\ 1 & 0
   \end{bmatrix}$.
    \end{enumerate}
\end{defn}

\begin{rem}\label{rem:line-plane}
\begin{enumerate}[label={\upshape(\roman*)}, align=right, widest=ii,  leftmargin=*]
\item[{\rm (1)}] Every unimodular free $R_2$-lattice of rank one is isomorphic to $\calL_1$. This
follows from the fact $\Nm(R_2^\times)=\Z_2^\times$, which can be
verified by a direct calculation. Alternatively, note that
$[O_{K_2}^\times: R_2^\times]\in \{1,3\}$ is always odd while
$\Z_2^\times$ is an abelian pro-2 group,  which implies that $
\Nm(R_2^\times)=\Nm(O_{K_2}^\times)=\Z_2^\times$. 
  
\item[{\rm (2)}]   Observe that  $\calH\not\simeq \calL_1^2$.  Indeed,  since
$\Tr_{K_2/\Q_2}(R_2)\subseteq 2\Z_2$, we have 
$\langle v, v\rangle\in 2\Z_2$ for every $v\in \calH$. The
same clearly fails for  $\calL_1^2$.
\end{enumerate}
\end{rem}

Conversely, the unimodular hyperbolic $R_2$-plane is characterized by the condition $\langle v, v\rangle\in 2\Z_2$ for all $v$'s as
shown by the following lemma.

\begin{lem}\label{lem:R2-hyperbolic}
Let $H$ be a unimodular hermitian free $R_2$-lattice
  of rank two.  If $\langle v, v\rangle\in 2\Z_2$ for all $v\in H$, then $H$ is isometric to the unimodular hyperbolic $R_2$-plane $\calH$. 
\end{lem}
\begin{proof}
Let   $e_1, e_2$ be an $R_2$-basis of $H$. 
Since $(H,
  \langle~,~\rangle)$ is unimodular, the Gram matrix $
  \begin{bmatrix}
    \langle e_1, e_1 \rangle & \langle e_1, e_2 \rangle\\ \langle e_2,
    e_1 \rangle & \langle e_2, e_2 \rangle
  \end{bmatrix}$  lies in $\GL_2(R_2)$. By our assumption, both
  $\langle e_1, e_1 \rangle$ and $\langle e_2, e_2 \rangle$ belong to
  $2\Z_2$, which then forces $\langle e_1, e_2 \rangle\in
  R_2^\times$.     Replacing $e_2$ by $\langle e_1, e_2 \rangle^{-1}e_2$
  if necessary, we may and will assume that $\langle e_1, e_2
  \rangle=1$.

Next, we produce an isotropic vector in
  $H\smallsetminus \grm_2 H$,  where $\grm_2=2O_{K_2}$ is the unique maximal
  ideal of $R_2$.   Suppose  that $\langle
  e_1, e_1\rangle=2a$ and $\langle e_2, e_2\rangle=2b$ for some
 $a, b\in \Z_2$. We show that the following
  equation in variable $x$ has a solution in $\Z_2$: 
  \begin{equation}
    \label{eq:16}
      \langle (x+x\sqrt{-p})e_1+e_2,  (x+x\sqrt{-p})e_1+e_2\rangle=0.  
  \end{equation}
The left hand side of the equation  simplifies into $f(x)\coloneqq
a(1+p)x^2+x+b\in \Z_2[x]$, which reduces to a monic linear
  polynomial modulo $2$.  Now it follows from Hensel's Lemma
  \cite[Chapter~II, (4.6)]{Neukirch-ANT} that $f(x)$ admits a monic
  linear factor in $\Z_2[x]$, or equivalently, the equation $f(x)=0$
  is solvable in $\Z_2$.  The solution $x_0\in \Z_2$ gives rise to  an
  isotropic vector $(x_0+x_0\sqrt{-p})e_1+e_2$ in
  $H\smallsetminus \grm_2 H$. 

Lastly, let $u\in H\smallsetminus \grm_2 H$ be an isotropic
vector. Since $(H,
  \langle~,~\rangle)$ is unimodular, there exists $v\in H\smallsetminus \grm_2 H$ such that
  $\langle u, v\rangle=1$. Since $\langle v, v\rangle\in 2\Z_2$, we
  can replace $v$ by $v-\frac{1}{2}\langle v, v\rangle u$ if necessary
  to achieve $\langle v, v\rangle=0$ without affecting the equality
  $\langle u, v\rangle=1$. In conclusion, we have found a pair of
  vectors $u, v\in H\smallsetminus \grm_2 H$ satisfying
    \[\langle u, u\rangle=0, \qquad \langle u, v\rangle=1, \qquad
      \langle v, v\rangle=0. \] Necessarily, $u, v$ forms an
    $R_2$-basis of $H$ and this yields the desired isomorphism
    $(H, \langle~,~\rangle)\simeq \calH$.
\end{proof}

\begin{prop}\label{prop:str-free-R2-latt}
  Let $N$ be a unimodular hermitian free $R_2$-lattice of rank $r$. 
  \begin{enumerate}
  \item[{\rm (1)}] If $r$ is odd, then $N\simeq \calL_1\perp
    \calH^{(r-1)/2}$. 
\item[{\rm (2)}] If $r$ is even, then $N$ is isometric to 
  either $\calH^{r/2}$ or $\calL_1^2\perp
  \calH^{(r-2)/2}$. 
  \end{enumerate}
\end{prop}
\begin{proof}
Let $\overline{N}\coloneqq N/\grm_2N$ be the reduction of $N$ modulo $\grm_2$. Then  $\overline{N}$ is a vector space of dimension $r$ over the finite field  $\F_2=R_2/2O_{K_2}$ of two elements. The perfect hermitian pairing $\langle~,~\rangle: N\times N\to R_2$ induces a nondegenerate symmetric $\F_2$-bilinear form on $\overline{N}$, which is still denoted by $\langle~,~\rangle$ by an abuse of notation. For example, $\overline{\calL}_1$ is an $\F_2$-line  $\F_2\bar{e}$ with $\langle \bar{e}, \bar{e}\rangle =1$, and $\overline{\calH}$ is  a hyperbolic plane $\F_2\bar{e}_1\oplus \F_2\bar{e}_2$ with 
\[\langle \bar{e}_1, \bar{e}_1\rangle =0, \qquad \langle \bar{e}_1, \bar{e}_2\rangle =1, \qquad \langle \bar{e}_2, \bar{e}_2\rangle =0. \]
The finite dimensional nondegenerate symmetric bilinear $\F_2$-spaces have been classified by the third named author in \cite[Lemma~4.7]{Yu-CF-PAMQ-2012}, where it is shown that every such space splits as an orthogonal direct sum of lines and hyperbolic planes.  More precisely, we have 
 \begin{enumerate}
        \item if $r$ is odd, then $\overline{N}\simeq \overline{\calL}_1\perp \overline{\calH}^{(r-1)/2}$;
        \item if $r$ is even, then $\overline{N}$ is isometric to either  $\overline{\calH}^{r/2}$ or $\overline{\calL}_1^2\perp \overline{\calH}^{(r-2)/2}$. 
    \end{enumerate}
Now we apply \cite[Corollary~I.3.3]{Baeza-QF-semilocal} to lift the orthogonal  splitting of $\overline{N}$ back to an orthogonal splitting of $N$.  Indeed, \cite[Corollary~I.3.3]{Baeza-QF-semilocal} shows that for every orthogonal decomposition $\overline{N}=\overline{P}\perp \overline{Q}$, there exists an orthogonal decomposition $N=P\perp Q$ of $N$ into unimodular $R_2$-sublattices with $P$ free over $R_2$, and $\overline{P}=P/\grm_2 P$, $\overline{Q}=Q/\grm_2 Q$.  
Apriori, \cite[Corollary~I.3.3]{Baeza-QF-semilocal} is only stated for quadratic or symmetric bilinear modules, but the same proof applies to the hermitian case as well.  Since $R_2$ is local, every finitely generated stably-free $R_2$-module is free. In particular, the $R_2$-lattice $Q$ is  free.  This allows us to apply \cite[Corollary~I.3.3]{Baeza-QF-semilocal} inductively to lift every line and hyperbolic plane in the orthogonal direct sum of  $\overline{N}$ into copies of  $\overline{\calL}_1$ and  $\overline{\calH}$ as above.  From Remark~\ref{rem:line-plane} (1), the lift of every $\F_2$-line is isomorphic to $\calL_1$. Similarly, Lemma~\ref{lem:R2-hyperbolic} shows that the lift of every hyperbolic $\F_2$-plane is isomorphic to $\calH$.    This proves the desire  decomposition of $N$ into orthogonal direct sums of copies of $\calL_1$ and $\calH$. 
\end{proof}

\begin{proof}[Proof of Proposition~\ref{prop:yu-bij}]
    Recall that  $\Bil_n$ denotes the set of isometric classes of symmetric $\F_2$-bilinear forms (including the degenerate ones) on the standard $n$-dimensional $\F_2$-space  $\F_2^n$. 
   It naturally decomposes into a disjoint union $\Bil_n=\coprod_{r=0}^n \Bil_n^r$ of subsets, where each  $\Bil_n^{r}\subset \Bil_n$ denotes the subset of  the classes with rank $r$.  For example, $\Bil_n^n$ consists of precisely the set of nondegenerate classes on $\F_2^n$.   
    Given  a symmetric bilinear form $b_r$ of rank $r$ on $\F_2^n$ , we write $R(b_r)$ for the kernel of $b_r$.   From~\cite[\S XV.1]{Lang-Algebra},  we have a canonical bijection $\Bil_n^r\to \Bil_r^r$ that sends each $[b_r]\in\Bil_n^r$ to the isometric class of  induced nondegenerate symmetric bilinear form $\bar{b}_r: \F_2^n/R(b_r)\times \F_2^n/R(b_r)\to \F_2$. 
    

    Similarly,the set $\Latt_n$  of isometric classes of $\Z_2$-rank $2n$ unimodular hermitian $R_2$-lattices also admits a natural decomposition $\Latt_n=\coprod_{r=0}^n \Latt_n^r$ into subset of classes of type $(r, n-r)$.  We have seen in Proposition~\ref{prop:orth-decomp-R2} that there is a canonical bijection $\Latt_n^r\to \Latt_r^r$ sending each class $[N]\in \Latt_n^r$ to the class $[N_1]\in \Latt_r^r$ for an orthogonal decomposition of $N=N_1\perp N_2$ with $N_1$ free of rank $r$ over $R_2$. On the other hand, it is shown in the proof of Proposition~\ref{prop:str-free-R2-latt} that reduction modulo $\grm_2$ induces a bijection $\Latt_r^r\to \Bil_r^r$ sending $[N_1]$ to $[N_1/\grm_2N_1]$.

     Now the desired bijection $\Latt_n\to \Bil_n$ is obtained by piecing together all the above bijections  for all $0\leq r\leq n$:
     \[\Latt_n=\coprod_{r=0}^n \Latt_n^r\cong \coprod_{r=0}^n \Latt_r^r\cong \coprod_{r=0}^n \Bil_r^r\cong \coprod_{r=0}^n \Bil_n^r=\Bil_n.\qedhere\]   
\end{proof}

Finally, we return to the classification of $\sqrt{-p}$-modular hermitian $R$-lattices in $(V, \varphi)$, where $(V, \varphi)$ is a positive-definite hermitian $K$-space of dimension $n$. In Lemma~\ref{lem:nec-cond-mod-latt},   we have listed some necessary conditions for the existence of such lattices in $(V, \varphi)$. We show that these necessary conditions are also sufficient and count the number of genera of  $\sqrt{-p}$-modular  $R$-lattices in $(V, \varphi)$.

\begin{thm}\label{thm:4.12}
    Keep the assumption that $p\equiv 3\pmod{4}$. \begin{enumerate}
        \item[{\rm (1)}] There exists a $\sqrt{-p}$-modular $R$-lattice $L$ in $(V, \varphi)$ if and only if 
    \begin{enumerate}[(a)]
        \item $p\equiv 7\pmod{8}$, $4\mid n$  and $d\varphi=[1]$; or  \label{item:i}
        \item $p\equiv 3\pmod{8}$, $4\mid n$ and $d\varphi=[1]$; or  \label{item:ii}
    \item   $p  \equiv 3\pmod{8}$, $ n\equiv 2\pmod{4}$  and $d\varphi=[2]$.\label{item:iii}
    \end{enumerate}
    \item[{\rm (2)}] Suppose that one of the conditions above holds so that there exists a $\sqrt{-p}$-modular $R$-lattice in $(V, \varphi)$. Then each genus of such lattices is uniquely determined by the corresponding isometric class of unimodular $R_2$-lattices in $(V_2, \varphi_2)$ at the prime $2$.  Moreover, there are $3n/2+1$ genera in case~\ref{item:i}, $n+1$ genera in case~\ref{item:ii}, and $n/2$ genera in case~\ref{item:iii}.    
    \end{enumerate}     
\end{thm}
\begin{proof}
   (1) First, suppose that conditions \ref{item:i} or \ref{item:ii} hold so that $4\mid n$ and $d\varphi=[1]$. From Lemma~\ref{lem:modular-OK-lat-v2},  there exists a $\sqrt{-p}$-modular  $O_K$-lattice in $(V, \varphi)$, so these two conditions are also sufficient. Next, suppose that condition~\ref{item:iii} holds.  In this case we identify $(V,\varphi)$ with the hermitian space $(K^n, \langle~,~\rangle)$, whose Gram matrix with respect to the standard basis is given by the diagonal matrix $\diag(1, \cdots, 1, 2)$. The standard $R$-lattice $L_0\coloneqq R^n$  in $(K^n, \langle~,~\rangle)$ is self-dual at every odd prime $\ell$. At the prime $p$,  condition~\ref{item:iii} implies that 
   \[(d\varphi, -p)_p=(2, -p)_p=-1=(-1)^{(\dim V)/2}, \]
   so it follows from  Lemma~\ref{lem:det-at-p} that  there exists a $\sqrt{-p}$-modular
  $O_{K_p}$-lattice $N_p$ in  $(K_p^n, \langle~,~\rangle_p)$. At the prime $2$, if we put $M_2\coloneqq \calL_1\perp \calL_0^{n-1}$, then $d(M_2\otimes_{\Z_2}\Q_2)=2^{n-1}$, from which it follows that we can identify $M_2\otimes \Q_2$ with $(V_2, \varphi_2)$ since they share the same determinant in $\Q_2^\times/\Nm(K_2^\times)$ as $n-1$ is odd. Now we apply the local-global
  principle of lattices \cite[Proposition~4.21]{curtis-reiner:1} to obtain an $R$-lattice $L$ in $K^n$ such that 
  \[L_p=N_p, \quad L_2=M_2, \quad\text{and}\quad L_\ell=L_0\otimes \Z_\ell \quad \text{for every odd prime $\ell\neq p$.} \]
Such an
  $R$-lattice $L$ is $\sqrt{-p}$-modular by construction. This prove the sufficiency of condition~\ref{item:iii}.


(2) Suppose that one of the three conditions in part (1) of the theorem holds. To classify the genera of $\sqrt{-p}$-modular  $R$-lattices in $(V, \varphi)$, observe that  $R_\ell$ is maximal at every odd prime $\ell$, and $p$ is the only ramified prime in $K/\Q$. Given two such $R$-lattices $L, L'$ in $(V, \varphi)$, we have $L_\ell\simeq L_\ell'$ for every odd prime $\ell\neq p$ by Lemma~\ref{lem:self-dual-OKl-lattice}, and $L_p\simeq L_p'$ by Lemma~\ref{lem:det-at-p}. Thus the classification of genera is reduced to purely local classification at the prime $2$. 
A unimodular $R_2$-lattice $L_2$ in $(V_2, \varphi_2)$ necessarily has type $(n-s, s)$ for some $0\leq s\leq n$. We treat the three different cases separately. 

First,  suppose that condition~\ref{item:i} holds. In this case $2$ is split in $K$ as $p\equiv 7\pmod{8}$, so there is a unique isometric class of nondegenerate hermitian $K_2$-space of rank $n$.  Combining Propositions~\ref{prop:orth-decomp-R2} and \ref{prop:str-free-R2-latt}, we see that 
\begin{equation}
    L_2\simeq \begin{dcases*}
        \calL_0^n & if $s=n$,\\
        \calL_1\perp \calH^{(n-s-1)/2}\perp \calL_0^s & if $s$ is odd,\\
        \calH^{(n-s)/2}\perp \calL_0^s \text{ or } \calL_1^2\perp\calH^{(n-s-2)/2}\perp \calL_0^s & if  $s$ is even and $s<n$. 
    \end{dcases*}
\end{equation}
It follows that there are $3n/2+1$ genera of $\sqrt{-p}$-modular $R$-lattices in $(V, \varphi)$ in this case. 

Next,  suppose that condition~\ref{item:ii} holds. In this case $2$ is inert in $K$, and $\Q_2/\Nm(K_2^\times)$ is a cyclic group of order $2$ generated by $2\Nm(K_2^\times)$.  The assumption $d\varphi=[1]$ then forces $s$ to be even for every unimodular $R_2$-lattice $L_2$ in $(V_2, \varphi_2)$. Indeed, if $s$ is odd, then $L_2\simeq \calL_1\perp \calH^{(n-s-1)/2}\perp \calL_0^s$, which 
leads to a contradiction since the determinant of $(\calL_1\perp \calH^{(n-s-1)/2}\perp \calL_0^s)\otimes \Q_2$ does not match with that of $(V_2, \varphi_2)$: 
\[ d((\calL_1\perp \calH^{(n-s-1)/2}\perp \calL_0^s)\otimes\Q_2)=(-1)^{(n-s-1)/2}2^s=2\neq d\varphi \in \Q_2/\Nm(K_2^\times). \]
Thus we have the following possibilities for the isometric class of $L_2$ in this case: 
\begin{equation}
    L_2\simeq \begin{dcases*}
        \calL_0^n & if $s=n$,\\
        \calH^{(n-s)/2}\perp \calL_0^s \text{ or } \calL_1^2\perp\calH^{(n-s-2)/2}\perp \calL_0^s & if  $s$ is even and $s<n$.
    \end{dcases*}
\end{equation}
It then follows that there are $n+1$ genera of $\sqrt{-p}$-modular $R$-lattices in $(V, \varphi)$.

Lastly, suppose that condition~\ref{item:iii} holds. The same argument as above shows that $s$ must be odd for this case. Therefore, every unimodular $R_2$-lattice $L_2$ in $(V_2, \varphi_2)$ is isometric to $\calL_1\perp \calH^{(n-s-1)/2}\perp \calL_0^s$ for some odd $s$ between $0$ and $n$. It follows that there are $n/2$ genera   of $\sqrt{-p}$-modular $R$-lattices in $(V, \varphi)$ in this case. 
\end{proof}

\section{unimodular lattices over commutative  local Bass orders}\label{sec:5}

Throughout this section, we fix a complete discrete valuation ring $Z$
and denote its field of fractions by $Q$. Let $E$ be a commutative
semi-simple  
$Q$-algebra, and $R$ be a  $Z$-order (of full rank) in
$E$. By an \emph{$R$-lattice} we mean a finite $R$-module $M$ in some
free $E$-module $V$ such that $M\otimes_Z Q=V$.  Suppose that $R$ is stable under a (possibly trivial) involution $a\mapsto a^\sigma$ of $E/Q$. Following \cite[\S XIII.7]{Lang-Algebra}, a \emph{sesquilinear form} on $M$ is a bi-additive pairing
\begin{equation}
  \label{eq:17}
\langle~,~\rangle:   M\times M \to R
\end{equation}
that is $R$-linear in its first variable and $(R,\sigma)$-linear in its second
variable. Let  $\varepsilon\in O_E^\times$ be a unit with $\varepsilon\varepsilon^\sigma=1$.  The sesquilinear form $\langle~,~\rangle$ is said to be \emph{$\varepsilon$-hermitian} if $\langle y, x\rangle=\varepsilon\langle x, y\rangle^\sigma$ for
all $x,y\in M$, in which case the unit $\varepsilon$ necessarily further satisfies $\varepsilon^\sigma\langle x, y\rangle\in R$ for all $x, y\in M$.  As usual, $1$-hermitian forms are simply called \emph{hermitian forms}, and $(-1)$-hermitian forms \emph{skew-hermitian forms}. If further the involution $\sigma$ is trivial, we obtain the usual symmetric and skew-symmetric bilinear forms. 



Return to the more general setting where $\langle~,~\rangle$ is sesquilinear.  
The form
$\langle~,~\rangle$ induces an $(R, \sigma)$-linear map
\begin{equation}
  \label{eq:18}
 \Phi: M\to \Hom_R(M, R), \qquad x \mapsto \langle \cdot, x \rangle.
\end{equation}
 If the above map $\Phi$ is
bijective, then we say that the sesquilinear form $\langle~,~\rangle$ is \emph{perfect}, or
equivalently, the sesquilinear $R$-lattice $(M, \langle~,~\rangle)$ is \emph{unimodular}; see \cite[Proposition~XIII.7.2]{Lang-Algebra}.  Very often the
form $\langle~,~\rangle$ will be clear from the context, so we drop it
from the notation and simply say that $M$ is a sesquilinear
 $R$-lattice.

The goal of this section is to provide an orthogonal decomposition theorem (namely, Theorem~\ref{thm:orthogonal-decomp}) for unimodular sesquilinear $R$-lattices. In classical literature such as those by
Baeza \cite{Baeza-QF-semilocal}, Knebusch \cite {Knebusch-1977} or
Knus \cite{Knus-1991}, it is customary to assume that $M$ is
projective (or equivalently, free over $R$ since here $R$ is
semilocal). However, practical applications such as the one considered
in the present paper call for more general study of such lattices
without the prior projective assumption. In this section, we show that
every unimodular lattice over a local \emph{Bass} order $R$ admits an
orthogonal decomposition such that each summand is essentially a free
unimodular lattice over an overorder of $R$. This reduces the study of
such lattices to the classical setting. See also Riehm
\cite{Riehm-1984} for a classification of hermitian lattices over (not
necessarily commutative) local hereditary orders.

Recall from \cite[\S37]{curtis-reiner:1} that a $Z$-order $R$ in
$E$ is \emph{Gorenstein} if $\Hom_Z(R, Z)$ is a projective $R$-module.
In an influential paper \cite{Bass-MathZ-1963}, Bass notes that
Gorenstein rings are ubiquitous, and he has also identified a subclass
of Gorenstein rings with many good structural properties that nowadays bear his name.
 Let $O_E$ be the maximal $Z$-order of
$E$, which exists as $Z$ is a complete discrete valuation ring. Equivalent characterizations of commutative
Bass orders in $E$ are given as follows 
\begin{enumerate}[(i)]
\item $R$ is a  \emph{Bass order}, that is,  every overorder
  $R'\supseteq R$ in $E$ is Gorenstein;
\item every ideal of $R$ can be generated by at most $2$-elements;
  \item the quotient $O_E/R$ is a cyclic $R$-module. 
\end{enumerate}
Characterization (ii) above is due to Bass \cite[\S7]{Bass-MathZ-1963}
himself, and (iii) is due to Borevich and Faddeev
\cite{MR0190187}, who  have
also obtained a classification theorem  \cite{MR0205980} of lattices over commutative
Bass orders (see also Curtis-Reiner~\cite[Section 37,
p.~789]{curtis-reiner:1}\footnote{The separability assumption on $E$ in \cite[Section 37,
p.~789]{curtis-reiner:1} is to ensure the existence of the maximal order $O_E$, which is the case in our situation. Thus, this assumption is not needed. Also see the equivalence of (i) and (iii) in \cite[2.1]{Levy-Wiegand-1985}.} for a summary). In the present setting, their
theorem  shows that every lattice $M$ over a Bass order $R$ determines an 
ascending chain of overorders
\begin{equation}
  \label{eq:20}
  R\subseteq R_1\subsetneq R_2 \subsetneq \cdots \subsetneq R_m, 
\end{equation}
and a decomposition of $M$ into a direct sum of $R$-submodules 
\begin{equation}
  \label{eq:19}
  M=M_1\oplus \cdots \oplus M_m
\end{equation}
such that each $M_i$ admits a free $R_i$-module structure of rank $r_i>0$.  The
chain of orders (\ref{eq:20}) and the $m$-tuple $(r_1, \cdots, r_m)\in
\Z_{>0}^m$ is uniquely determined by the $R$-isomorphism class of $M$, so henceforth they will be called the \emph{structural
invariants} of $M$. However, it should be noted that the $M_i$ themselves are not
uniquely determined by $M$.  Nevertheless,  each $M_i$ will be
called an \emph{isotypic  component} of $M$, so being an isotypic
component of $M$ is equivalent to being a direct summand of 
 $M$ that admits a free $R_i$-module structure of rank $r_i$ for some
 $1\leq i \leq m$. 
For each $1\leq i \leq m$, let $e_1^{(i)},
\cdots, e_{r_i}^{(i)}$  be an $R_i$-basis of $M_i$. We collect these
bases together and call the ordered set
\begin{equation}
  \label{eq:26}
\scrB_M\coloneqq
\{e_1^{(1)}, \cdots, e_{r_1}^{(1)}, \cdots, e_1^{(m)}, \cdots,
e_{r_m}^{(m)}\}  
\end{equation}
a \emph{pseudo-basis} of $M$. Clearly, $\scrB_M$ forms
an $E$-basis of the free $E$-module $M\otimes_Z Q$.


From the lifting of idempotents
\cite[Corollary~7.6]{Eisenbud-Com-alg}, if $R$ is not local, then it
factors as a direct product $R_1\times R_2$ of
$Z$-orders. Characterization (iii)  above by  Borevich and Faddeev then shows that $R$ is Bass
if and only if both $R_1$ and $R_2$ are Bass. Thus, the study of unimodular sesquilinear $R$-lattices can be reduced to the case where $R$ is local.
For simplicity,  we
assume that $R$ is a local Bass order for the rest of this
section. From the classification of Drozd-Kirichenko-Roiter
\cite[\S IX, Theorems~6.5 and 6.14, and Lemmas~6.6 and 6.16]{Roggenkamp-Latt-II}, the localness
assumption on $R$ implies the following properties about the algebra $E$ and
the chain of orders in \eqref{eq:20}:
\begin{itemize}
\item $E$ is either a field or a product $E_1\times E_2$ of two
  fields;
\item  each $R_i$ is local for $1\leq i\leq m-1$, and $R_m$ is not
  local if and only if $E=E_1\times E_2$ and $R_m=O_{E_1}\times
  O_{E_2}$;
\item if $R$ is stable under a nontrivial involution $a\mapsto
  a^\sigma$ of $E/Q$, then so is each $R_i$ by the uniqueness of the
  minimal overorder;
  
\item if $E=E_1\times E_2$, then each $R_i$ is a  \emph{subdirect sum}
  of $O_{E_1}\times O_{E_2}$, meaning that  the canonical projection
  maps $\pr_j: R_i\to O_{E_j}$ is surjective for every  $j\in \{1,2\}$. 
\end{itemize}


To study unimodular sesquilinear $R$-lattices, it is necessary to understand the
structure of $\Hom_R(M,R)$.  We shall prove in
Corollary~\ref{cor:pseudo-basis-dual} that $\Hom_R(M,R)$ is always
isomorphic to $M$ and write down a pseudo-basis of it. From the
decomposition \eqref{eq:19}, to show $\Hom_R(M,R)\simeq M$, it is enough
to show that $\Hom_R(R', R)$ is a free $R'$-module of rank one for
every overorder $R'\supseteq R$ in $E$.  For this purpose let us 
recall the classical \emph{colon-quotient operation} of fractional
$R$-ideals.  By definition, a fractional $R$-ideal $\gra$ is a finitely
generated $R$-submodule of $E$ with $\gra\otimes_Z Q=E$. 
Given two fractional $R$-ideals $\gra, \grb$ in $E$, we
put
\begin{equation}
  \label{eq:21}
  (\gra:\grb)\coloneqq \{x\in E\mid x\grb\subseteq\gra\}, 
\end{equation}
which is canonically isomorphic to $\Hom_R(\grb, \gra)$.  It follows from the definition that $((\gra:\grb):\grc)=(\gra: \grb \grc)$.

\begin{defn}
The
\emph{relative conductor} $\grf_{R'/R}$ of an overorder $R'\supseteq
R$  is defined as
\[\grf_{R'/R}\coloneqq \Ann_R(R'/R)=(R: R')\simeq \Hom_R(R', R).\]    
\end{defn}

Clearly, $\grf_{R'/R}$ is the largest ideal of $R'$ contained in $R$,
and is also the largest ideal of $R$ which is an ideal of $R'$.
\begin{lem}\label{lem:rel_cond}
  The relative conductor $\grf_{R'/R}$ is a principal $R'$-ideal, that
  is, there exists an $\alpha'\in E^\times\cap R$ such that
  $\grf_{R'/R}=R'\alpha'$. In particular, $\Hom_R(R', R)\simeq R'$. 
\end{lem}
\begin{proof}
By our assumption $R$ is Bass, so 
$R'$ is Gorenstein. From \cite[Corollary~2.7]{Brzezinski-loc-Princ}, 
to show that $\grf_{R'/R}$ is a principal
$R'$-ideal,  it is enough
to show that $(\grf_{R'/R}: \grf_{R'/R})=R'$, that is, the $R'$-ideal $\grf_{R'/R}$ is $R'$-proper. Now we calculate
directly to obtain
\[
  \begin{split}
  (\grf_{R'/R}: \grf_{R'/R})&=((R: R'): \grf_{R'/R})=(R:
                              R'\grf_{R'/R})\xeq{(\dagger)}(R:
                              \grf_{R'/R})     \\
    &=(R: (R: R'))\xeq{(\ddagger)} R'. 
  \end{split}
\]
Here $(\dagger)$ holds since $\grf_{R'/R}$ is an $R'$-ideal, and
$(\ddagger)$ holds because $R$ is a Gorenstein order of Krull
dimension $1$, so every fractional $R$-ideal (particularly, $R'$) is
reflexive by the original work of Bass
\cite[Theorem~6.3]{Bass-MathZ-1963} (see also \cite[Characterization A 2.6]{Gorenstein-orders-JPAA-2015}).
\end{proof}

 By the assumption, $R$ is stable under the involution $\sigma\in \Aut(E/Q)$. We have seen that $R'$ is also $\sigma$-stable, which in turn implies that $\grf_{R'/R}$ is $\sigma$-stable as well.  Thus for any  $R'$-generator $\alpha'$ of $\grf_{R'/R}$, we have 
    \begin{equation}\label{eqn:5.7}
        (\alpha')^\sigma=u'\alpha' \qquad \text{for some}\quad  u'\in R'^\times \quad \text{satisfying} \quad u'(u')^\sigma=1. 
    \end{equation}
\begin{lem}
     There exists a $\sigma$-invariant generator  of $\grf_{R'/R}$ if one of the following conditions holds: 
 \begin{enumerate}[(i)]
     \item[{\rm (i)}] the Galois cohomology group $H^1(\langle \sigma \rangle, R'^\times)$ is trivial;
     \item[{\rm (ii)}] $\sigma\neq \id$ and $R$ is an $O_F$-order, where $F=E^\sigma$ is the fixed subfield of $\sigma$. 
 \end{enumerate}
\end{lem}
\begin{proof}
    (i) Choose an arbitrary generator $\alpha'$ of $\grf_{R'/R}$ as above. From \eqref{eqn:5.7}, the unit $u'\in R'^\times$ determines a unique $1$-cocycle of $\langle \sigma \rangle$ into  $R'^\times$,  whose cohomology class depends only on  $\grf_{R'/R}$
    and not on the choice of the generator $\alpha'$.
    If $H^1(\langle \sigma \rangle, R'^\times)=\{1\}$, then there exists $v'\in R'^\times$ such that $u'=v'^{-1}(v')^\sigma$.  Now $\beta'\coloneqq v'^{-1}\alpha'$ is a $\sigma$-invariant  generator of $\grf_{R'/R}$.

   (ii) In this case,   every $O_F$-suborder $R\subseteq O_E$ is of the form $O_F+\grf O_E$ for some $O_F$-ideal $\grf$. Another $O_F$-order $R'=O_F+\grf'O_E$ contains $R$ if and only if $\grf'\mid \grf$, in which case the relative conductor $\grf_{R'/R}$ is given by $(\grf'^{-1}\grf) R'$.  Now any $O_F$-generator of $\grf'^{-1}\grf$ is a $\sigma$-invariant generator of $\grf_{R'/R}$. 
\end{proof}

\begin{ex}
In general there may not exist any $\sigma$-invariant generator of $\grf_{R'/R}$, as shown by the following example. Let  $E\coloneqq\Q_2(\sqrt{-1}, \sqrt{-3})$, and $\sigma\in\Gal(E/\Q_2)$ be the unique nontrivial automorphism of $E$ over the subfield $F\coloneqq \Q_2(\sqrt{-3})$. The order $R\coloneqq \Z_2+\Z_2\sqrt{-1}+\Z_2\sqrt{-3}+\Z_2(1+\sqrt{-1})(1+\sqrt{-3})/2$ is a $\sigma$-stable suborder of $O_E$  of index $2$.   If we put $L\coloneqq \Q_2(\sqrt{-1})$ and $\grp\coloneqq (1+\sqrt{-1})O_L$, then $R=O_L+\grp O_E$.  In particular, we have $(R:O_E)=\grp O_E$, which cannot have a $\sigma$-invariant generator since $E/F$ is ramified. 
\end{ex}

Keep the  generator $\alpha'$ of $\grf_{R'/R}$ fixed. Note that each $R'$-lattice $M'$ can also be viewed as an
$R$-lattice, so we have two different kinds of sesquilinear forms on $M'$, namely,
those $R'$-valued ones and the $R$-valued ones. If $\langle~,~\rangle':M'\times M'\to R'$ is an $R'$-valued form, then clearly $\alpha'\langle~,~\rangle$ takes value in $R$.
The following lemma
shows that all $R$-valued pairings on $M'$ arise in this way. 
\begin{cor}\label{cor:bij-pairing}
For each $R'$-lattice $M'$, multiplication by $\alpha'$ induces a bijection between the
following two sets of sesquilinear
forms:
\begin{equation}\label{eq:22}
\left\{\parbox{3.2cm}{$R'$-valued form \newline $\langle~,~\rangle': M'\times
    M'\to R'$}\right\}\quad \xleftrightarrow[\langle~,~\rangle'  \leftrightarrow  \alpha' \langle~,~\rangle']{1:1}\quad
  \left\{\parbox{3cm}{$R$-valued form \newline $\langle~,~\rangle: M'\times
  M'\to R$}\right\}. 
\end{equation}
Moreover, this bijection enjoys the following two properties:
\begin{enumerate}[(i)]
    \item[{\rm (i)}] $\langle~,~\rangle': M'\times
    M'\to R'$ is a perfect pairing of $R'$-lattices if and only if the
    corresponding $\alpha'\langle~,~\rangle': M'\times
    M'\to R$ is a perfect pairing of $R$-lattices; 
    \item[{\rm (ii)}] $\langle~,~\rangle'$ is $\varepsilon'$-hermitian (for some $\varepsilon'\in O_E^\times$ with $\varepsilon'(\varepsilon')^\sigma=1$) if and only if the corresponding  pairing $\alpha'\langle~,~\rangle'$ is $(u'^{-1}\varepsilon')$-hermitian, where $u'=(\alpha')^{-1}(\alpha')^\sigma\in R'^\times$ as in \eqref{eqn:5.7}. 
\end{enumerate}
In particular, if $(\alpha')^\sigma=\alpha'$, then the correspondence $\langle~,~\rangle'  \leftrightarrow  \alpha' \langle~,~\rangle'$  establishes  a bijection between hermitian forms. 
\end{cor}
\begin{proof}
  As explained above, multiplication by $\alpha'$ induces a injection
  from the left hand side to the right hand side. Conversely, suppose
  that $\langle~,~\rangle: M'\times M'\to R$ is an $R$-valued pairing on
  $M'$ (regarded as an $R$-lattice). For each element $x'\in M'$, the
  set $\langle M', x'\rangle$ form an $R'$-ideal contained in $R$, but
  we have seen that $\grf_{R'/R}=R'\alpha'$ is the maximal $R'$-ideal
  contained in $R$, so the pairing $\langle~,~\rangle: M'\times M'\to R$
  actually takes values in $R'\alpha'$.  Extending by ($\sigma$-)linearity from
  $R$ to $R'$, we obtain an $R'$-valued pairing
  $\alpha'^{-1}\langle~,~\rangle: M'\times M'\to R'$ on the $R'$-lattice
  $M'$. This establishes the desired bijection \eqref{eq:22}.  Property (ii) of the bijection follows from a straightforward calculation. 

  To prove property (i) of the bijection, let
  $\langle~,~\rangle': M' \times M'\to R'$ be an $R'$-valued pairing on
  the $R'$-lattice $M'$.  We write
  $\Phi': M'\to \Hom_{R'}(M', R')$ (resp.~$\Phi: M'\to \Hom_R(M', R)$)
  for the map induced by $\langle~,~\rangle'$
  (resp.~by $\langle~,~\rangle\coloneqq \alpha'\langle~,~\rangle'$) as
  defined in   \eqref{eq:18}.  By construction, $\Phi=\alpha'\Phi'$. 
  Note that $\Phi'$ determines a unique $(E,\sigma)$-linear map 
  \[ M'\otimes_Z
  Q\to \Hom_E(M'\otimes_Z Q, E), \]  
and 
  both $\Hom_{R'}(M', R')$ and
  $\Hom_R(M', R)$ can be regarded canonically as $R'$-submodules of
  $\Hom_E(M'\otimes_Z Q, E)$.  We claim that 
  \[ \Hom_R(M',
  R)=\alpha'\Hom_{R'}(M', R'), \] 
  from which it then follows that
  $\Phi'$ is a bijection if and only if $\Phi$ is a bijection. Once
  again, the
  claim is a direct consequence of the fact that $\grf_{R'/R}=R'\alpha'$ is the maximal $R'$-ideal
  contained in $R$, since we have 
  \[
    \begin{split}
      \Hom_R(M', R)&=\{\lambda\in \Hom_E(M'\otimes_ZQ, E)\mid
                     \lambda(M')\subseteq R\}\\
      &=\{\lambda\in \Hom_E(M'\otimes_ZQ, E)\mid
        \lambda(M')\subseteq \alpha' R'\}\\
      &=\alpha'\cdot \{\lambda\in \Hom_E(M'\otimes_ZQ, E)\mid
        \lambda(M')\subseteq  R'\}\\&=\alpha'\Hom_{R'}(M', R'). 
    \end{split}
  \]
This verifies the claim and completes the proof of the corollary. 
\end{proof}

\begin{cor}\label{cor:pseudo-basis-dual}
Let $M$ be an $R$-lattice whose structural invariants are given by
the ascending chain \eqref{eq:20} and the $m$-tuple
$(r_1, \cdots, r_m)\in \Z_{>0}^m$. For each $1\leq i \leq m$, fix a
generator $\alpha_i$ of $(R:R_i)$.   
  Let $\scrB_M=\{e^{(i)}_k\mid 1\leq i \leq
 m , 1\leq k\leq r_i\}$ be a pseudo-basis of $M$ as in \eqref{eq:26}, and
 $\{f_k^{(i)}\}$ be the corresponding dual $E$-basis in $\Hom_E(M\otimes
 Q, E)$. In other words, we have
\[f_k^{(i)}(e_l^{(j)})=
  \begin{cases}
    1 & \text{if } i=j \text{ and } k=l;\\
    0 & \text{otherwise}.     
  \end{cases}
\]
Then a pseudo-basis of the $R$-lattice  $\Hom_R(M,R)$ is given by
\[ \scrC_M\coloneqq \{\alpha_i
f^{(i)}_k\mid 1\leq i\leq  m, 1\leq k\leq r_i\}. \] 
\end{cor}

\begin{proof}
  By the definition of pseudo-basis, we have $M=\oplus_{i=1}^m\oplus_{k=1}^{r_i}
  R_ie_k^{(i)}$.  Thus
  \[
    \begin{split}
    \Hom_R(M,R)&=\oplus_{i=1}^m\oplus_{k=1}^{r_i}
    \Hom_R(R_ie_k^{(i)},
                  R)\\&=\oplus_{i=1}^m\oplus_{k=1}^{r_i}(R:R_i)f_k^{(i)}=
      \oplus_{i=1}^m\oplus_{k=1}^{r_i} R_i\alpha_if_k^{(i)}.
    \end{split}
  \]
  This clearly shows that $\scrC_M=\{\alpha_if^{(i)}_k\}$ form a
  pseudo-basis of $\Hom_R(M,R)$. 
\end{proof}

Keep the notation and assumption of
Corollary~\ref{cor:pseudo-basis-dual}, and fix a sesquilinear form $\langle~,~\rangle: M\times
M\to R$ on $M$ for the rest of this section. 
 We form its Gram matrix $A$ with respect to the pseudo-basis
$\scrB_M$ as follows: $A=[A_{ij}]_{1\leq i, j\leq m}$ is defined to be
a block
matrix, where each $A_{ij}=[a_{kl}^{(ij)}]$
is an $(r_i\times r_j)$-matrix with $(k, l)$-entry
$a^{(ij)}_{kl}\coloneqq \langle e^{(i)}_k , e^{(j)}_l  \rangle$. As both
$\langle R_ie^{(i)}_k , e^{(j)}_l  \rangle$ and $\langle e^{(i)}_k ,
R_je^{(j)}_l  \rangle$ are contained in $R$, we find that
\begin{equation}
  \label{eq:24}
  a_{kl}^{(ij)}\coloneqq \langle e^{(i)}_k , e^{(j)}_l  \rangle\in
 R_i\alpha_i\cap R_j\alpha_j= R_t\alpha_t \qquad \text{with}\quad t\coloneqq \max\{i, j\}. 
\end{equation}
Indeed, if $i\leq j$, then $R_i\subseteq R_j$, and hence $  R_i\alpha_i=(R: R_i)\supseteq (R: R_j)=R_j\alpha_j$.
Now a direct calculation yields the following lemma. 
\begin{lem}\label{lem:exp-matrix}
  Let $\Phi: M\to Hom_R(M, R)$ be the map induced by
  $\langle~,~\rangle: M\times M\to R$ as defined in \eqref{eq:18}. Let
  $\scrB_M=\{e_k^{(i)}\}$ and $\scrC_M=\{\alpha_if_k^{(i)}\}$ be the
  pseudo-bases of $M$ and $\Hom_R(M, R)$ respectively given by
  Corollary~\ref{cor:pseudo-basis-dual}. Then the matrix
  of $\Phi$ with respective to these two bases is given by
  $B=[\alpha_i^{-1}A_{ij}]_{1\leq i,j\leq m}$,
  where $A=[A_{ij}]_{1\leq i, j\leq m}$ is the Gram matrix of the
  sesquilinear form   $\langle~,~\rangle$ with respect to $\scrB_M$
  defined  above.
\end{lem}

For simplicity, we put
$n\coloneqq \sum_{i=1}^m r_i=\rk_E(M\otimes_Z Q)$.  Let
$\Lambda_0\coloneqq\oplus_{i=1}^m R_i^{r_i}$ be the $R$-lattice in
$E^n$ with the same structural invariants as $M$. We  regard the matrix $B$ as an element of
$\End_R(\Lambda_0)$, which is an $R$-suborder of $\Mat_n(O_E)$. By
definition, the sesquilinear form $\langle~,~\rangle: M\times M\to R$ is perfect
if and only if $B\in \Aut(\Lambda_0)$. On the other hand, since $O_E$
is a finite $R$-algebra, a similar proof as the one for
\cite[Proposition~2.6]{xue-yang-yu:ECNF} shows that
\begin{equation}
  \label{eq:27}
  \Aut(\Lambda_0)=\End_R(\Lambda_0)\cap \GL_n(O_E).  
\end{equation}
Let $\grJ(O_E)$ denote the Jacobson radical of $O_E$, and $\bar{B}\in
\Mat_n(O_E/\grJ(O_E))$ be the image of $B$ modulo
$\grJ(O_E)$. We have the following equivalent characterizations
about invertibility of $B$: 
\begin{equation}
  \label{eq:28}
  B\in \Aut(\Lambda_0)\quad \iff \quad B\in
  \GL_n(O_E) \quad \iff \quad \bar{B}\in
  \GL_n(O_E/\grJ(O_E)).   
\end{equation}
 Note that for each pair $(i,j)$ with
$1\leq i < j \leq m$, the $(i,j)$-block $B_{ij}$ of $B$ defines an
element of $\Hom_R(R_j^{r_j}, R_i^{r_i})$, so every entry of $B_{ij}$
lies in $(R_i: R_j)$. This is compatible with our previous
description of entries of $B=[\alpha_i^{-1}A_{ij}]_{1\leq i,j\leq m}$ given in \eqref{eq:24} since 
\begin{equation}
  \label{eq:25}
  \begin{split}
  (R_i: R_j)&=\alpha_i^{-1}((R: R_i): R_j)=\alpha_i^{-1}(R:
  R_iR_j)\\&=\alpha_i^{-1}(R: R_j)=\alpha_i^{-1}\alpha_jR_j.     
  \end{split}
\end{equation}
\begin{lem}\label{lem:upper-block-radical}
We have $(R_i: R_j)\subseteq \grJ(O_E)$ for every pair $(i,j)$ with
$1\leq i < j \leq m$. In particular, $B_{ij}\subseteq
\Mat_{r_i\times r_j}(\grJ(O_E))$ for every such pair $(i, j)$. 
\end{lem}
\begin{proof}
First note that   $(R_i: R_j)$ is a proper ideal of
$R_j$ since $R_i\subsetneq R_j$. 
  From \cite[Exercise~5.5.ii)]{Atiyah-Mac},  the
Jacobson radical $\grJ(R_j)$ coincides with $\grJ(O_E)\cap
R_j$. If $R_j$ is local, then
$\grJ(R_j)$ is just the unique maximal ideal of $R_j$. In particular,
if $R_j$ is local, then $(R_i: R_j)\subseteq \grJ(R_j)\subseteq
\grJ(O_E)$. 

For the remaining proof assume that $R_j$ is not local. This 
happens only when $j=m$, $E=E_1\times E_2$ and $R_m=O_E=O_{E_1}\times
O_{E_2}$. The ideals of $O_E$ not contained in $\grJ(O_E)$ are of
the form $O_{E_1}\times \gra_2$ or $\gra_1\times O_{E_2}$, where
$\gra_1$ and $\gra_2$ are 
ideals of $O_{E_1}$  and $O_{E_2}$ respectively. On the other
hand, $(R_i: R_j)$ is also an ideal of $R_i$, so if $(R_i:
R_j)\not\subseteq \grJ(O_E)$, then $R_i$ necessarily contains a
nontrivial idempotent, contradicting to the fact that $R_i$ is
always local. In conclusion, we have shown that $(R_i:
R_j)\subseteq \grJ(O_E)$ whenever $1\leq i<j\leq m$. 
\end{proof}

\begin{prop}\label{prop:restr-unimodular}
  Let $M=M_1\oplus \cdots \oplus M_m$ be a decomposition of an $R$-lattice
  into isotypic components as in \eqref{eq:19}, 
  and
  $\langle~,~\rangle_i:M_i\times M_i\to R$ be  the restriction of the
  given pairing $\langle~,~\rangle$ to $M_i$. 
  The following are
  equivalent:
  \begin{enumerate}[(i)]
  \item[{\rm (i)}] $\langle~,~\rangle: M\times M \to R$ is a perfect pairing of
    $R$-lattices;
    \item[{\rm (ii)}] $\langle~,~\rangle_i: M_i\times M_i \to R$ is a perfect pairing of
      $R$-lattices for every $1\leq i\leq m$;
     \item[{\rm (iii)}]  $\alpha_i^{-1}\langle~,~\rangle_i: M_i\times M_i \to R_i$ is a perfect pairing of
      $R_i$-lattices for every $1\leq i\leq m$. 
  \end{enumerate}
\end{prop}
\begin{proof}
  The equivalence of (ii) and (iii) follows from
  Corollary~\ref{cor:bij-pairing},  so we focus on proving the
  equivalence between (i) and (iii).  Recall that the
  pseudo-basis $\scrB_M=\{e^{(i)}_k\mid 1\leq i \leq
 m , 1\leq k\leq r_i\}$ is constructed by picking an $R_i$-basis $\{e_1^{(i)},
\cdots, e_{r_i}^{(i)}\}$  for each $M_i$ and collecting them together. Let $B=[B_{ij}]_{1\leq i,
  j\leq m}$ be the matrix of $\Phi: M\to \Hom_R(M, R)$ with
respect to the pseudo-bases $\scrB_M$ and $\scrC_M$ as above. Then
each $B_{ij}=\alpha_i^{-1}[\langle e^{(i)}_k,
e^{(j)}_l\rangle]_{\substack{\tiny 1\leq
k\leq r_i\\\tiny 1\leq l\leq r_j}}$. In particular, $B_{ii}$ coincides
with the Gram matrix of the  pairing
$\alpha_i^{-1}\langle~,~\rangle_i: M_i\times M_i \to R_i$ with respect
to the $R_i$-basis $\{e_1^{(i)},
\cdots, e_{r_i}^{(i)}\}$ of $M_i$.  We have seen in \eqref{eq:28} that
$\langle~,~\rangle: M\times
M \to R$ is a perfect pairing if and only if $B\in \GL_n(O_E)$, while
by definition each $\alpha_i^{-1}\langle~,~\rangle_i: M_i\times
M_i \to R_i$ is a perfect pairing if and only if $B_{ii}\in
\GL_{r_i}(R_i)$ for each $1\leq i\leq m$. Thus the proof of the proposition is reduced to showing that
$B\in \GL_n(O_E)$ if and only if each $B_{ii}\in \GL_{r_i}(R_i)$ for
every $1\leq i\leq m$.

On the other hand, we have seen in
Lemma~\ref{lem:upper-block-radical} that $B_{ij}\subseteq
\Mat_{r_i\times r_j}(\grJ(O_E))$ for every pair $(i, j)$ with $1\leq
i<j\leq m$. Reducing $B$ modulo $\grJ(O_E)$, we obtain a block
lower triangular matrix $\bar{B}\in \Mat_n(O_E/\grJ(O_E))$, whose diagonal blocks $\bar{B}_{ii}\in
\Mat_{r_i}(O_E/\grJ(O_E))$  are
just  the images of the $B_{ii}$'s modulo $\grJ(O_E)$. Therefore we have the following equivalences
\[
  \begin{split}
 B\in \Aut(\Lambda_0)\quad & \iff \quad  B\in \GL_n(O_E)\quad \iff \quad \bar{B}\in
  \GL_n(O_E/\grJ(O_E))\\ & \iff \quad  \bar{B}_{ii}\in
                           \GL_{r_i}(O_E/\grJ(O_E)),  \quad \forall\, 1\leq i \leq m;\\
                           &\iff \quad  B_{ii}\in
                             \GL_{r_i}(O_E), \quad \forall\, 1\leq i \leq m;\\
                             &\iff \quad B_{ii}\in
                             \GL_{r_i}(R_i), \quad\forall\, 1\leq i \leq m.
  \end{split}
\]
This finishes the proof of equivalence between (i) and (iii). 
\end{proof}

\begin{lem}\label{lem:orthog-decomp}
  Let  $N\subseteq M$ be an $R$-sublattice, and  $\langle~,~\rangle_N: N\times N\to R$ 
  be the restriction of the sesquilinear form $\langle~,~\rangle: M\times M \to
  R$ to $N$. Suppose that $\langle~,~\rangle_N$ is a perfect pairing
  on $N$. Then  there is an orthogonal decomposition
  \begin{equation}
    \label{eq:15}
  M=N\oplus N^\perp,     
  \end{equation}
  where $N^\perp\coloneqq \{v\in M\mid \langle N, v\rangle=0\}$ is the right orthogonal complement of $N$. 
  Moreover,  $(M, \langle~,~\rangle)$ is unimodular if and only if $(N^\perp,
  \langle~,~\rangle_{N^\perp})$ is unimodular. 
\end{lem}
\begin{proof}
  This is a slight generalization of
  \cite[Proposition~I.3.2]{Baeza-QF-semilocal} or
  \cite[Proposition~I.2]{Knebusch-1977}, and the same proofs apply here. 
\end{proof}


Now we are ready to state the main theorem of this section. For the
reader's convenience, we also recall  briefly the basic setup of this
section. 

\begin{thm}\label{thm:orthogonal-decomp}
Let $R$ be a local Bass order, and $M$ be an $R$-lattice whose structural invariants are given by
the ascending chain $R_1\subsetneq R_2\subsetneq \cdots \subsetneq
R_m$ of overorders of $R$ and the $m$-tuple
$(r_1, \cdots, r_m)\in \Z_{>0}^m$. For each $1\leq i \leq m$, fix a
generator $\alpha_i$ of $(R:R_i)$.   
  Suppose that $\langle~,~\rangle:M\times M\to R$ is a perfect
  sesquilinear form
  on the $R$-lattice $M$. Then there exists an orthogonal decomposition of
  $M$ into isotypic components
  \begin{equation}
    \label{eq:29}
  M=M_1\oplus \cdots\oplus M_m
  \end{equation}
  such that $\langle M_i, M_j\rangle=0$ for all $1\leq i<j\leq m$. 
  In particular, each $M_i$ is a free $R_i$-lattice of rank $r_i$, and  the restriction $\langle~,~\rangle_{M_i}: M_i\times
M_i\to R$ of $\langle~,~\rangle$ to $M_i$ is of the form
$\alpha_i\langle~,~\rangle_i$, where $\langle~,~\rangle_i: M_i\times
M_i\to R_i$ is a  perfect pairing of $R_i$-lattices for each $1\leq
i\leq m$.

\end{thm}
\begin{proof}[Proof by induction on $m$]
  We focus on proving the existence of the orthogonal
  decomposition  \eqref{eq:29}, as the last statement of the theorem
  follows directly from    Corollary~\ref{cor:bij-pairing}  and Proposition~\ref{prop:restr-unimodular}.
   If $m=1$, then there is no decomposition to be made.

Now suppose that such an orthogonal decomposition exists for any
unimodular $R$-lattice with $m-1$
  isotypic components. Let $(M, \langle~,~\rangle)$ be a unimodular
  $R$-lattice with $m\geq 2$ isotypic
  components. Replacing the pairing
  by $\alpha_1^{-1}\langle~,~\rangle$ if necessary, we assume that
  $R_1=R$. Let $M=N_1\oplus \cdots \oplus N_m$ be an arbitrary
  decomposition of $M$ into isotypic components and put
  $M_1\coloneqq N_1$, which is a free $R$-lattice of rank $r_1$. Since
  $\langle~,~\rangle: M\times M\to R$ is a perfect pairing, 
the
  restriction $\langle~,~\rangle_{M_1}: M_1\times M_1\to R$ is again perfect
 by  Proposition~\ref{prop:restr-unimodular}. Now Lemma~\ref{lem:orthog-decomp} allows us to decompose
  $M$ as the  direct sum of $M_1$ and   $M_1^\perp$, and the restriction
  $\langle~,~\rangle_{M_1^\perp}: M_1^\perp\times M_1^\perp\to R$ is
  again perfect.  Since $M_1^\perp$ is an $R$-lattice with $m-1$
  isotypic components, we  apply the induction hypothesis to
  $(M_1^\perp, \langle~,~\rangle_{M_1^\perp})$ to obtain an orthogonal
  decomposition of $M_1^\perp$ into isotypic components
  $M_1^\perp=M_2\oplus \cdots \oplus M_m$ such that $\langle M_i, M_j\rangle=0$ for all $2\leq i<j\leq m$.  Putting this together with
  the original decomposition $M=M_1\oplus M_1^\perp$, we obtain the
  desired orthogonal decomposition of $M$. 
\end{proof}

Suppose further that $\langle~,~\rangle$ is $\varepsilon$-hermitian for some $\varepsilon\in O_E^\times$ with $\varepsilon\varepsilon^\sigma=1$. Then Theorem~\ref{thm:orthogonal-decomp} produces an orthogonal decomposition $M=M_1\perp\cdots \perp M_m$ in the classical sense, namely, $M=M_1\oplus\cdots \oplus M_m$ with $\langle M_i, M_j\rangle=0$ for all $i\neq j$. From the proof of Theorem~\ref{thm:orthogonal-decomp}, it is clear
that such a  decomposition is not unique. A natural
question to ask is that given two orthogonal decompositions 
\[ M=M_1\perp \cdots\perp M_m=N_1\perp \cdots\perp N_m,\] with both
$M_i$ and $N_i$ being free $R_i$-lattices of rank $r_i$ for all
$1\leq i\leq m$, do we necessarily have
$(M_i,\langle~,~\rangle_{M_i})\simeq (N_i, \langle~,~\rangle_{N_i})$
for all $i$?  Unfortunately, the answer is negative in
general, as shown by the example below. 
\begin{ex}\label{ex:no-witt}
Take
$E=\Q_2(\sqrt{-2})$, and $R=\Z_2[2\sqrt{-2}]=\Z_2+2O_E$.  Let
$(M, \langle~,~\rangle)$ be the hermitian $R$-lattice
$Re_1\oplus O_Ee_2\subseteq E^2$ with the hermitian pairing given by
\[\langle e_1, e_1\rangle=1, \qquad \langle e_1, e_2\rangle=0, \qquad
  \langle e_2, e_2\rangle=2.\]
From Proposition~\ref{prop:restr-unimodular}, the pairing
$\langle~,~\rangle: M\times M\to R$ is perfect. 
The following are  two orthogonal decompositions of $M$: 
\[M=Re_1\perp 
O_Ee_2=R(e_1+e_2)\perp O_E(-2e_1+e_2).\]
Since $\langle e_1+e_2, e_1+e_2\rangle =3\not\in \Nm(R^\times)$, we
find that 
\[(Re_1, \langle~,~\rangle_{Re_1})\not\simeq (R(e_1+e_2), \langle~,~\rangle_{R(e_1+e_2)}).\]
\end{ex}

\section*{Acknowledgments}
The present work was prepared during the last author's visits to Wuhan University. He thanks his host Xue and Wuhan University for the warm
hospitality and great working conditions.   
Xue is partially supported by the National Natural Science Foundation of China grant No.~12271410 and No.~12331002. Yu is partially supported
by the grants NSTC 114-2115-M-001-001 and AS-IA-112-M01.   

\bibliographystyle{hplain}
\bibliography{TeXBiB}
\end{document}